\documentclass{amsart}

\usepackage{mathtools}
\usepackage{float}
\usepackage{amsmath}
\usepackage{amsfonts}
\usepackage{amssymb}
\usepackage{graphicx}
\usepackage{hyperref}
\usepackage{mathrsfs}
\usepackage{pb-diagram}
\usepackage{epstopdf}
\usepackage{amsmath,amsfonts,amsthm,enumerate,amscd,latexsym,curves}
\usepackage{bbm}
\usepackage{epsfig,epsf,xypic,epic}
\usepackage{caption}
\usepackage{subcaption}
\usepackage{tikz}
\usepackage{geometry}
\usetikzlibrary{matrix,arrows,decorations.pathmorphing}
\usetikzlibrary{cd}
\usepackage{tikz-cd}
\usepackage{enumitem}
\makeatletter
\newcommand{\xslalph}[1]{\expandafter\@xslalph\csname c@#1\endcsname}
\newcommand{\@xslalph}[1]{
    \ifcase#1\or a\or b\or c\or \or d\or e\or f\or g\or h\or i
    \or j\or k\or l\or m\or n\or o\or p\or r\or s\or 
    \or t\or u\or v\or z\or \v{z}
    \else\@ctrerr\fi
}
\AddEnumerateCounter{\xslalph}{\@xslalph}{m}
\makeatother

\usetikzlibrary{matrix,arrows,decorations.pathmorphing}
\theoremstyle{theorem}
\newtheorem{theorem}{Theorem}[section]

\newtheorem{lemma}[theorem]{Lemma}
\newtheorem{corollary}[theorem]{Corollary}
\newtheorem{proposition}[theorem]{Proposition}

\theoremstyle{definition}
\newtheorem{definition}[theorem]{Definition}

\newtheorem{notation}[theorem]{Notation}

\newtheorem{assumption}[theorem]{Assumption}

\newtheorem{example}[theorem]{Example}

\newtheorem{remark}[theorem]{Remark}
\newtheorem{conjecture}[theorem]{Conjecture}

\numberwithin{equation}{section}

\newcommand{\Int}{\mathrm{Int}}
\newcommand{\Id}{\mathrm{Id}}
\newcommand{\grad}{\mathrm{grad}}

\newcommand{\im}{\mathrm{Im}}
\newcommand{\re}{\mathrm{Re}}

\newcommand{\hol}{\mathrm{hol}}

\newcommand{\trop}{\mathrm{trop}}

\newcommand{\Ext}{\mathrm{Ext}}

\newcommand{\tw}{\mathrm{tw}}
\newcommand{\sflat}{\mathrm{sf}}
\newcommand{\diag}{\mathrm{diag}}

\newcommand{\Hom}{\mathrm{Hom}}

\newcommand{\pd}[2]{\frac{\partial #1}{\partial #2}}
\newcommand{\secpd}[3]{\frac{{\partial}^2 #1}{{\partial #2}{\partial #3}}}

\newcommand{\inner}[1]{\left\langle #1\right\rangle}

\newcommand{\bb}[1]{\mathbb{#1}}
\newcommand{\ring}[1]{\mathring{#1}}
\newcommand{\cu}[1]{\mathcal{#1}}
\newcommand{\til}[1]{\widetilde{#1}}

\newcommand{\ol}[1]{\overline{#1}}

\newcommand{\mf}[1]{\mathfrak{#1}}

\newcommand{\Spec}{\mathrm{Spec}}

\def\Hom{\text{\rm Hom}}

\begin{document}
	
	\title[]{Toric vector bundles, non-abelianization, and spectral networks}

	\author[Suen]{Yat-Hin Suen}
	\address{Department of Mathematics\\ National Cheng Kung University\\No. 1, Dasyue Rd.\\ Tainan City 70101\\ Taiwan}
	\email{yhsuen@gs.ncku.edu.tw}

	\date{\today}

	\begin{abstract}
		Spectral networks and non-abelianization were introduced by Gaiotto-Moore-Neitzke and they have many applications in mathematics and physics. In a recent work by Nho, he proved that the non-abelianization of an almost flat local system over the spectral curve of a meromorphic quadratic differential is the same as the family Floer construction. Based on the mirror symmetry philosophy, it is then natural to ask how holomorphic vector bundles arise from spectral networks and non-abelianization. In this paper, we construct toric vector bundles on complete toric surfaces via spectral networks and non-abelianization arising from Lagrangian multi-sections. As an application, we deduce that the moduli space of rank 2 toric vector bundles over toric surfaces admit an $A$-type $\cu{X}$-cluster structure.
	\end{abstract}

	\maketitle
	
	\tableofcontents
	
\section{Introduction}

Let $X_{\Sigma}$ be the toric variety associated to a rational complete fan $\Sigma$ in a vector space $N_{\bb{R}}$. Denote by $T$ the dense algebraic torus in $X_{\Sigma}$. We assume $X_{\Sigma}$ is projective. The $T$-equivariant mirror of $X_{\Sigma}$ is given by the pair $(T^*M_{\bb{R}},\Lambda_{\Sigma})$, where
$$\Lambda_{\Sigma}:=\bigcup_{\tau\in\Sigma}(\tau^{\perp}+M)\times(-\tau)\subset T^*M_{\bb{R}}=M_{\bb{R}}\times N_{\bb{R}}.$$
This is a (singular) conical Lagrangian subset of $T^*M_{\bb{R}}$ with respect to the standard symplectic structure. By identifying $T^*M_{\bb{R}}$ with the interior of the closed unit disk bundle $D^*M_{\bb{R}}$ centred at the origin, we obtain a (singular) Legendrian subset
$$\Lambda_{\Sigma}^{\infty}:=\ol{\Lambda_{\Sigma}}\cap\partial(D^*M_{\bb{R}}).$$
The \emph{equivariant homological mirror symmetry} \cite{CCC_HMS, OS} stated that there is a quasi-equivalence
$$\cu{F}_{FLTZ}:\cu{F}uk(T^*M_{\bb{R}},\Lambda_{\Sigma})\xrightarrow{\simeq}\cu{P}erf_T(X_{\Sigma}),$$
where $\cu{F}uk(T^*M_{\bb{R}},\Lambda_{\Sigma})$ is the Fukaya category of exact/tautologically unobstructed Lagrangian submanifolds with $\Lambda_{\Sigma}$-admissibility condition $\bb{L}^{\infty}\subset\Lambda_{\Sigma}^{\infty}$ and $\cu{P}erf_T(X_{\Sigma})$ is the category of $T$-equivariant perfect complexes on $X_{\Sigma}$. We call $\cu{F}_{FLTZ}$ the \emph{FLTZ mirror functor}. In \cite{CCC_HMS}, the authors showed that $\cu{F}_{FLTZ}$ is compatible with SYZ \cite{SYZ} in the sense that it maps Lagrangian sections to toric line bundles. This has been generalized by Oh and the author of this paper in \cite{OS} that $\cu{F}_{FLTZ}$ maps Lagrangian multi-sections to toric vector bundles by using the microlocal characterization of toric vector bundles given in \cite{Morse_theory_TVB}.

In \cite{Suen_trop_lag}, the author introduced the notion of \emph{tropical Lagrangian multi-sections}\footnote{This tropical notion first appeared in \cite{branched_cover_fan, Suen_TP2} and was generalized to arbitrary integral affine manifolds with singularities equipped with polyhedral decomposition in \cite{CMS_k3bundle, Suen_TLMS_TLFS} to study the open reconstruction problem in the Gross-Siebert program \cite{GS1,GS2,GS11}} and studied the \emph{$B$-realization problem}. An $r$-fold tropical Lagrangian multi-section $\bb{L}^{\trop}$ over $\Sigma$ is an $r$-fold branched covering map $p:(L^{\trop},\Sigma_{L^{\trop}},\mu)\to(N_{\bb{R}},\Sigma)$ between cone complexes together with a piecewise linear function $\varphi^{\trop}:L^{\trop}\to\bb{R}$ with integral slopes. In \cite{branched_cover_fan}, Payne associated to a rank $r$ toric vector bundle $\cu{E}\to X_{\Sigma}$ an $r$-fold tropical Lagrangian multi-section $\bb{L}_{\cu{E}}^{\trop}$ over $\Sigma$, which we regard as the tropicalization of $\cu{E}$. The $B$-realization problem asks the converse of this assignment, namely, the existence of a toric vector bundle $\cu{E}$ whose tropicalization is prescribed. On the other hand, \cite{OS} studied the mirror counterpart. A tropical Lagrangian multi-section $\bb{L}^{\trop}$ determines a Lagrangian subset
$$\Lambda_{\bb{L}^{\trop}}:=\bigcup_{\tau'\in\Sigma_{L^{\trop}}}m(\tau')\times(-p(\tau'))\subset\Lambda_{\Sigma},$$
where $m(\tau')\in M/(\tau^{\perp}\cap M)$ is the slope of $\varphi^{\trop}$ along $\tau'$, and we view it as a coset of $\tau^{\perp}\cap M$. The \emph{$A$-realization problem} asks if there exists an unobstructed (in the sense of \cite{FOOO1, AJ}) $r$-fold Lagrangian multi-section $\bb{L}$ in $T^*M_{\bb{R}}$ for which it is \emph{$\Lambda_{\bb{L}^{\trop}}$-admissible}, i.e. $\bb{L}^{\infty}\subset\Lambda_{\bb{L}^{\trop}}^{\infty}$. The case of Lagrangian sections was solved in \cite{Abouzaid09} by Abouzaid. In dimension 2, Oh and the author provided a complete solution to the A-realizability of a 2-fold tropical Lagrangian multi-section \cite{OS}. The $A$-realization problem implies the $B$-realization problem via applying equivariant homological mirror symmetry. See also \cite{Matessi_Lag_pants, Mikhalkin_trop_to_Lag, Mak_Ruddat_trop_Lag_CY, Hicks_realization, Hicks_Trop_Lag_hyperseuface_unobs, Hicks_toric_del_Pezzo} for non-locally free realization problems.

\emph{Spectral networks} and \emph{non-abelianization} were introduced in \cite{Spectral_networks}. Given an $r$-fold branched covering map of real orientable surfaces $p:L\to C$ and a rank 1 local system $\cu{L}$ on $L$, one wants to define $p_*\cu{L}$ as a rank $r$ locally system on $C$. However, the naive pushforward $p_*\cu{L}$ receives non-trivial monodromies around the branch locus. A spectral network, roughly, is a collection of paths on $C$ that satisfies some combinatorial rules. These rules guide us on how to resolve these monodromies to result in a rank $r$ local system on $C$. The rank $r$ local system obtained in this way is called the \emph{non-abelianization of $\cu{L}$ with respect to the spectral network}. By the definition of spectral networks, non-abelianization gives more than a local system on $C$. It comes with natural decorations on the boundaries/punctures of $C$, called the \emph{flag data}. As we will see in Section \ref{sec:mirror_construction}, these flag data are the key to relating non-abelianizations, which are local systems, with toric vector bundles, which are coherent sheaves.

In \cite{Nho_spectral_network_family_fleor}, when $L\subset T^*C$ is the spectral curve of a meromorphic differential on $C$, Nho showed that non-abelianization of an almost-flat local system $\cu{L}$ on $L$ is the same as the \emph{family Floer assignment} $x\mapsto HF((L,\cu{L}),T_x^*C)$. Motivated by the work of Nho and mirror symmetry, it is very natural to ask how spectral networks and non-abelianization contribute to vector bundles. The main goal of this paper is to give \emph{explicit} construction of toric vector bundles on complete toric surfaces via spectral networks and non-abelianization. The main result of this paper is the following

\begin{theorem}[=Theorem \ref{thm:TVB_via_SN}]
    Let $\bb{L}^{\trop}$ be a tropical Lagrangian multi-section over $\Sigma$ and $\Delta$ be a polytope Legendre dual to the fan. Suppose there exists a spectral network $\cu{W}_{\bb{L}^{\trop}}$ on $\Delta$ subordinate to a simply branched covering map $p:L\to\Delta$ with flag data given by \eqref{eqn:slope_condition}. Then for any $\Bbbk^{\times}$-local system $\cu{L}$ on $L$, there exists a toric vector bundle $\cu{E}(\cu{W}_{\bb{L}^{\trop}},\cu{L})$ on $X_{\Sigma}$ with tropicalization $\bb{L}^{\trop}$. Furthermore, the assignment $\Psi_{\cu{W}}:\cu{L}\mapsto\cu{E}(\cu{W}_{\bb{L}^{\trop}},\cu{L})$ gives an injection from $\cu{M}(L,GL_1)$, the moduli space of rank 1 local systems on $L$, to $\cu{M}(X_{\Sigma},\bb{L}^{\trop})$, the moduli space of toric vector bundles with tropicalization $\bb{L}^{\trop}$.
\end{theorem}

To construct such spectral network, we begin with a 2-dimensional $\Lambda_{\bb{L}^{\trop}}$-admissible Lagrangian multi-section $\bb{L}$ in $T^*M_{\bb{R}}$, i.e. $\bb{L}^{\infty}\subset\Lambda_{\bb{L}^{\trop}}^{\infty}$, and define the notion of \emph{spectral network subordinate to $\bb{L}$} (Definition \ref{def:spectral_network_L}). Roughly speaking, such spectral network is obtained by applying Legendre transform to the flow lines of the gradient equations
\begin{equation}\label{eqn:grad_intro}
    \dot{\xi}(t)=\grad_{\ol{g}_{\phi}}(\varphi^{(\alpha)}-\varphi^{(\beta)})(\xi(t))
\end{equation}
with $\alpha,\beta=1,\dots,r$ being distinct, that emit from branched points and their intersection points. Here, $\varphi^{(\alpha)}$'s are local primitives of the Lagrangian multi-section $\bb{L}$ and $\ol{g}_{\phi}$ is some Riemannian metric on $N_{\bb{R}}$ that is flat around the branch points and is of Hessian-type outside a compact subset of $N_{\bb{R}}$. This gives a collection of walls $\cu{W}_{\bb{L}}$ on the polytope $\Delta$ of the toric variety. We will prove a key lemma (Lemma \ref{lem:boundary_wall}), which ensures that if $\cu{W}_{\bb{L}}$ is a spectral network, then it always respect the flag data given in \eqref{eqn:slope_condition}. As a consequence, we obtain

\begin{corollary}[=Corollary \ref{cor:TVB_via_SN}]
    Let $\bb{L}$ be an $r$-fold $\Lambda_{\bb{L}^{\trop}}$-admissible Lagrangian multi-section and suppose $\cu{W}_{\bb{L}}$ is a non-degenerated spectral network. Then for any rank 1 local system $\cu{L}$ on the domain $L$, there exists a rank $r$ toric vector bundle $\cu{E}(\bb{L},\cu{L})$ whose tropicalization is $\bb{L}^{\trop}$ and the assignment $\Psi_{\cu{W}_{\bb{L}}}:\cu{L}\mapsto\cu{E}(\bb{L},\cu{L})$ is an injection.
\end{corollary}

Because of the combinatorial rules of spectral networks given in \cite{Spectral_networks}, $\cu{W}_{\bb{L}}$ may not be a spectral network for general $\bb{L}$. Hence we should also ask ourselves, for what $\bb{L}$, a (non-degenerated) spectral network subordinate to $\bb{L}$ exists. We will introduce the notion of a \emph{well-behaved} (Definition \ref{def:well_behaved}) Lagrangian multi-sections. Straightly speaking, well-behavedness says that gradient lines of \eqref{eqn:grad_intro} can only terminate on $\partial\Delta$ after the Legendre transform. Let us also mention that a generic 2-fold Lagrangian multi-section with index 1 immersed double points is well-behaved. In particular, \cite[Theorem 5.2]{OS} provides many well-behaved Lagrangian multi-sections. This generic condition ensures $\cu{W}_{\bb{L}}$ is a spectral network.

\begin{theorem}[=Theorem \ref{thm:existence}]
    If $\bb{L}$ is a well-behaved $\Lambda_{\bb{L}^{\trop}}$-admissible Lagrangian multi-section with simple branching, then there exists a non-degenerated spectral network subordinate to $\bb{L}$.
\end{theorem}

Combining Corollary 1.2 and Theorem 1.3, we obtain

\begin{corollary}[=Corollary \ref{cor:well_behave_TVB}]
    If $\bb{L}$ is a well-behaved $r$-fold $\Lambda_{\bb{L}^{\trop}}$-admissible Lagrangian multi-section, then for any $\Bbbk^{\times}$-local system $\cu{L}$ on the domain $L$, there exists a rank $r$ toric vector bundle $\cu{E}(\bb{L},\cu{L})$ whose tropicalization is $\bb{L}^{\trop}$ and the assignment $\Psi_{\cu{W}_{\bb{L}}}:\cu{L}\mapsto\cu{E}(\bb{L},\cu{L})$ is an injection.   
\end{corollary}

In Section \ref{sec:existence}, we will give some examples on non-degenerated spectral networks based the work by Oh and the author \cite{OS} and give two applications to our construction, one on the $A$-side and one on the $B$-side. The first application is an enhancement on \cite[Corollary 5.16]{OS}.

\begin{theorem}[=Theorem \ref{thm:A_real}]
    A 2-fold tropical Lagrangian multi-section can be realized by a connected embedded exact Lagrangian multi-section in $T^*M_{\bb{R}}$ if and only if $N_{\bb{L}^{\trop}}\geq 3$.
\end{theorem}

The number $N_{\bb{L}^{\trop}}$ was introduced to any 2-dimensional 2-fold tropical Lagrangian multi-section $\bb{L}^{\trop}$ by Oh and the author in \cite{OS}, where we proved that $N_{\bb{L}^{\trop}}\geq 3$ is a sufficient condition for $\bb{L}^{\trop}$ to be realized by an embedded Lagrangian multi-section, and when $N_{\bb{L}^{\trop}}$ is odd, the converse is true by applying mirror symmetry. Theorem 1.3 proves that $N_{\bb{L}^{\trop}}\geq 3$ is always necessary for embedded A-realizability \emph{without} using mirror symmetry.

The second application is the discovery of cluster structure on the moduli space $\cu{M}(X_{\Sigma},\bb{L}^{\trop})$ of rank 2 toric vector bundles over toric surface $X_{\Sigma}$ with fixed tropicalization $\bb{L}^{\trop}$. Cluster structure was introduced by Fomin-Zelevinsky \cite{FZ_Cluster1, FZ_Clustert2, BFZ_Cluster3}. Gaiotto-Moore-Neitzke pointed out (at least when $r=2$) in \cite{Spectral_networks} cluster transformations can be understood as ``connecting" two non-degenerated spectral networks via a degenerated one. We, therefore, expect $\Psi_{\cu{W}_{\bb{L}}}:\cu{M}(L,GL_1)\to\cu{M}(X_{\Sigma},\bb{L}^{\trop})$ are cluster charts.

\begin{theorem}[=Theorem \ref{thm:cluster}]
    Let $\bb{L}^{\trop}$ be a 2-dimensional 2-fold tropical Lagrangian multi-section with $N_{\bb{L}^{\trop}}\geq 3$. Then the moduli space $\cu{M}(X_{\Sigma},\bb{L}^{\trop})$ is an $A_{N_{\bb{L}^{\trop}}-3}$-type $\cu{X}$-cluster variety.
\end{theorem}

\subsection*{Organization} We begin by reviewing the notions of toric vector bundles, tropical Lagrangian multi-sections, and Lagrangian multi-sections in Section \ref{sec:prelim}. We then introduce the notion of spectral networks subordinate to a $\Lambda_{\bb{L}^{\trop}}$-admissible Lagrangian multi-section in Section \ref{sec:spec_net_L} and prove Theorem 1.3. We will recall the construction of the non-abelianization map in Section \ref{sec:non_abel} and prove Theorem 1.1 in Section \ref{sec:mirror_construction}. In Section \ref{sec:existence}, we give several examples of spectral networks subordinate to those 2-fold Lagrangian multi-sections constructed in \cite{OS}, one non-example, and one 3-fold example. We end this paper by proving Theorem 1.3 and 1.4.

\subsection*{Acknowledgment}
The author would like to thank the anonymous referees for their useful comments and suggestions. The author is grateful to Yoon Jae Nho for his comments and fruitful discussions, and would also like to thank Mark Gross and Yong-Geun Oh for their interest in this work.

\section{Toric vector bundles via spectral networks and non-abelianization}

Fix a field $\Bbbk$. Let $N=\bb{Z}^n$, $M=\Hom_{\bb{Z}}(N,\bb{Z})$, $N_{\bb{R}}=N\otimes_{\bb{Z}}\bb{R}$, and $M_{\bb{R}}=M\otimes_{\bb{Z}}\bb{R}$. Let $\Sigma$ be a complete rational fan on $N_{\bb{R}}$ and $X_{\Sigma}$ be the associated complete toric variety over $\Bbbk$. Denote the dense algebraic torus in $X_{\Sigma}$ by $T$ and for a cone $\sigma\in\Sigma$, we denote by $U(\sigma):=\Spec(\Bbbk[\sigma^{\vee}\cap M])$ the corresponding affine chart of $X_{\Sigma}$.

\subsection{Objects of interest}\label{sec:prelim}

We begin by introducing the three objects that we are interested in this work.

\begin{definition}
    A vector bundle $\cu{E}$ on $X_{\Sigma}$ is called \emph{toric} if the $T$-action on $X_{\Sigma}$ lifts to an action on $\cu{E}$ which is linear on fibers.
\end{definition}

Let $\sigma\in\Sigma$ be a maximal cone. On the affine chart $U(\sigma)$, $\cu{E}$ admits a $T$-equivariant trivialization $\cu{E}|_{U(\sigma)}\cong\cu{O}_{U(\sigma)}^{\oplus r}$ so that with respect to an equivariant frame $\{1^{(\alpha)}(\sigma)\}_{\alpha=1}^r$, the $T$-action is given by
$$t \cdot 1^{(\alpha)}(\sigma)=t^{m^{(\alpha)}(\sigma)}1^{(\alpha)}(\sigma),$$
for some $m^{(\alpha)}(\sigma)\in M$, $\alpha=1,\dots,r$, and $t\in T$. The $T$-equivariant structure constrains the transition maps of $\cu{E}$. Indeed, if $\sigma_1,\sigma_2$ are two maximal cones in $\Sigma$, then with respective to the above $T$-equivariant frames, the transition map from the affine chart $U(\sigma_1)$ to the affine chart $U(\sigma_2)$ is given by
$$G_{\sigma_1\sigma_2}:=\left(g_{\sigma_1\sigma_2}^{(\alpha\beta)}z^{m^{(\alpha)}(\sigma_1)-m^{(\beta)}(\sigma_2)}\right)_{\alpha\beta},$$
where $g_{\sigma_1\sigma_2}^{(\alpha\beta)}$ are constants in the underlying field subordinate to the condition that $g_{\sigma_1\sigma_2}^{(\alpha\beta)}\neq 0$ only if $$m^{(\alpha)}(\sigma_1)-m^{(\beta)}(\sigma_2)\in(\sigma_1\cap\sigma_2)^{\vee}\cap M.$$

\begin{definition}
    A \emph{tropical Lagrangian multi-section $\bb{L}^{\trop}$ over a fan $\Sigma$} is a branched covering map $p:(L^{\trop},\Sigma_{L^{\trop}},\mu)\to(N_{\bb{R}},\Sigma)$ of connected cone complexes \cite[Definition 2.17]{branched_cover_fan} together with a piecewise linear function $\varphi^{\trop}:L^{\trop}\to\bb{R}$ with rational slopes, where $\mu:\Sigma_{L^{\trop}}\to\bb{Z}_{>0}$ is the multiplicity map. A cone $\tau'\in\Sigma_{L^{\trop}}$ is said to be \emph{ramified} if $\mu(\tau')>1$. A tropical Lagrangian is said to be \emph{separated} if for any ray $\rho\in\Sigma(1)$ and distinct lifts $\rho^{(\alpha)},\rho^{(\beta)}\in\Sigma_{L^{\trop}}(1)$ of $\rho$, we have $\varphi^{\trop}|_{\rho^{(\alpha)}}\neq \varphi^{\trop}|_{\rho^{(\beta)}}$.
\end{definition}

In \cite{branched_cover_fan}, Payne associated a tropical Lagrangian multi-section $\bb{L}_{\cu{E}}^{\trop}$ over $\Sigma$ to a toric vector bundle $\cu{E}\to X_{\Sigma}$, which we now recall. Let $\sigma\in\Sigma$ and $U(\sigma)$ be the corresponding affine chart. The toric vector bundle splits equivariantly on $U(\sigma)$ as
$$\cu{E}|_{U(\sigma)}\cong\bigoplus_{m(\sigma)\in\textbf{m}(\sigma)}\cu{L}_{m(\sigma)},$$
where $\textbf{m}(\sigma)\subset M(\sigma):=M/(\sigma^{\perp}\cap M)$ is a multi-set and $\cu{L}_{m(\sigma)}$ is the line bundle corresponds to the linear function $m(\sigma)\in M(\sigma)$. Payne's construction is as follows. Let $|\Sigma|\to\Sigma$ be the map given by mapping $x\in|\Sigma|$ to the unique cone $\sigma\in\Sigma$ such that $x\in \Int(\sigma)$. Equip $\Sigma$ with the quotient topology. Define
$$\Sigma_{\cu{E}}:=\{(\sigma,m(\sigma))\,|\,\sigma\in\Sigma, m(\sigma)\in\textbf{m}(\sigma)\}$$
and let $\Sigma_{\cu{E}}\to\Sigma$ be the projection 
$$(\sigma,m(\sigma))\mapsto\sigma.$$
Equip $\Sigma_{\cu{E}}$ a poset structure
$$(\sigma_1,m(\sigma_1))\leq (\sigma_2,m(\sigma_2))\Longleftrightarrow\sigma_1\subset\sigma_2\text{ and }m(\sigma_2)|_{\sigma_1}=m(\sigma_1)$$
and equip it with the poset topology. Define
$$L_{\cu{E}}^{\trop}:=|\Sigma|\times_{\Sigma}\Sigma_{\cu{E}}.$$
Let the set of cones on $L_{\cu{E}}^{\trop}$ be $\Sigma\times_{\Sigma}\Sigma_{\cu{E}}\cong\Sigma_{\cu{E}}$. The multiplicity $\mu_{\cu{E}}:L_{\cu{E}}^{\trop}\to\bb{Z}_{>0}$ is defined by
$$\mu_{\cu{E}}(\sigma,m(\sigma)):=\text{number of times that }m(\sigma)\text{ appears in }{\bf{m}}(\sigma).$$
The projection map $p_{\cu{E}}:L_{\cu{E}}^{\trop}\to|\Sigma|$ then induces a rank $r$ branched covering map of cone complexes $p_{\cu{E}}:(L_{\cu{E}}^{\trop},\Sigma_{\cu{E}},\mu_{\cu{E}})\to(N_{\bb{R}},\Sigma)$. The piecewise linear function $\varphi_{\cu{E}}^{\trop}:L_{\cu{E}}^{\trop}\to\bb{R}$ is tautologically given by
$$\varphi_{\cu{E}}^{\trop}|_{(\sigma,m(\sigma))}:=p_{\cu{E}}^*m(\sigma).$$
This gives a tropical Lagrangian multi-section $\bb{L}_{\cu{E}}^{\trop}:=(L_{\cu{E}}^{\trop},\Sigma_{\cu{E}},\mu_{\cu{E}},p_{\cu{E}},\varphi_{\cu{E}}^{\trop})$ and we call it the \emph{tropicalization of $\cu{E}$}. It was also shown in \cite{Suen_trop_lag} that the tropicalization of a toric vector bundle must be separated. From now on, we assume all tropical Lagrangian multi-sections in this paper are separated and all ramified cones are of codimension 2.

As mentioned in the introduction, given a tropical Lagrangian multi-section $\bb{L}^{\trop}$ over $\Sigma$, the $B$-realization problem asks if there exits a toric vector bundle $\cu{E}\to X_{\Sigma}$ so that $\bb{L}_{\cu{E}}^{\trop}=\bb{L}^{\trop}$. For each cone $\tau'\in\Sigma_{L^{\trop}}$, we denote by $m(\tau')\in M/(p(\tau')^{\perp}\cap M)$ the slope of $\varphi^{\trop}$ along $\tau'$. The tropical Lagrangian multi-section $\bb{L}^{\trop}$ is B-realizable if we can find a collection of constants $(g_{\sigma_1^{(\alpha)}\sigma_2^{(\beta)}})_{\sigma_1^{(\alpha)},\sigma_2^{(\beta)}\in\Sigma_{L^{\trop}}(n)}$ so that
$$G_{\sigma_1\sigma_2}:=\left(g_{\sigma_1^{(\alpha)}\sigma_2^{(\beta)}}z^{m(\sigma_1^{(\alpha)})-m(\sigma_2^{(\beta)})}\right)_{\alpha\beta}\in GL(r,\Bbbk[(\sigma_1\cap\sigma_2)^{\vee}\cap M])$$
and satisfy the cocycle condition
$$G_{\sigma_3\sigma_1}G_{\sigma_2\sigma_3}G_{\sigma_1\sigma_2}=\Id.$$
The data $(g_{\sigma_1^{(\alpha)}\sigma_2^{(\beta)}})_{\sigma_1^{(\alpha)},\sigma_2^{(\beta)}\in\Sigma_{L^{\trop}}(n)}$ is called a \emph{Kaneyama's data} as it was first introduced by Kaneyama \cite{Kaneyama_classification} to classify toric vector bundles on toric varieties.

We now move to the $A$-side. By the SYZ philosophy \cite{SYZ, LYZ, CS_SYZ_imm_Lag}, the mirror of vector bundles should be realized by Lagrangian multi-sections. We recall that by choosing a metric, one can identify $T^*M_{\bb{R}}$ with the interior of the closed unit disk bundle $D^*M_{\bb{R}}$ and for any subset $S\subset T^*M_{\bb{R}}$, we define $S^{\infty}:=\ol{S}\cap\partial(D^*M_{\bb{R}})$.

\begin{definition}
    An \emph{$r$-fold Lagrangian multi-section $\bb{L}$ of $T^*M_{\bb{R}}$} is a graded exact (with respect to the canonical 1-form $\lambda_{M_{\bb{R}}}$ on $T^*M_{\bb{R}}$) Lagrangian immersion $i:L\to T^*M_{\bb{R}}$ such that the composition $p_L:=p_{N_{\bb{R}}}\circ i:L\to N_{\bb{R}}$ is a branched covering map of degree $r$, where $p_{N_{\bb{R}}}:T^*M_{\bb{R}}=M_{\bb{R}}\times N_{\bb{R}}\to N_{\bb{R}}$ is the natural projection. By abuse of notation, we use $\bb{L}$ to denote the image $i(L)$. Denote by $S\subset N_{\bb{R}}$ and $S'\subset L$ the branch locus and ramification locus of $p_L$, respectively. Let $\bb{L}^{\trop}$ be a tropical Lagrangian multi-section over a complete fan $\Sigma$ and define
    $$\Lambda_{\bb{L}^{\trop}}:=\bigcup_{\tau'\in\Sigma_{L^{\trop}}}m(\tau')\times(-p(\tau'))\subset\Lambda_{\Sigma}.$$
    A Lagrangian multi-section is said to be \emph{$\Lambda_{\bb{L}^{\trop}}$-admissible} if $\bb{L}^{\infty}\subset\Lambda_{\bb{L}^{\trop}}^{\infty}$.
\end{definition}

Given a tropical Lagrangian multi-section $\bb{L}^{\trop}$ over a complete fan $\Sigma$. The \emph{$A$-realization problem} asks if there exists a tautologically unobstructed $\Lambda_{\bb{L}^{\trop}}$-admissible Lagrangian multi-section. Note that if we don't require unobstructedness, one can easily construct many $\Lambda_{\bb{L}^{\trop}}$-admissible immersed Lagrangian multi-sections in $T^*M_{\bb{R}}$ by smoothing the corners of $\varphi^{\trop}$.

By using the microlocal criterion for toric vector bundles given in \cite{Morse_theory_TVB}, the authors of \cite{OS} showed that a tautologically unobstructed $\Lambda_{\bb{L}^{\trop}}$-admissible Lagrangian multi-section (with certain grading) has equivariant mirror being a toric vector bundle over $X_{\Sigma}$ whose tropicalization is $\bb{L}^{\trop}$. Moreover, they also solved the $A$-realization problem for 2-fold tropical Lagrangian multi-sections over any 2-dimensional complete fan hence solving the $B$-realization problem via the A-realizability of $\bb{L}^{\trop}$ in this case by applying equivariant homological mirror symmetry. However, the construction of the mirror bundle is not explicit. The main goal of this paper is to construct Kaneyama's data explicitly via spectral networks and non-abelianization. The resulting toric vector bundle is expected to be compatible with equivariant mirror symmetry.

\subsection{Spectral networks subordinate to Lagrangian multi-sections}\label{sec:spec_net_L}

We recall the definition of spectral network subordinate to a branched covering map $p:L\to C$ of orientable surfaces. %To fit into our content, we allow $\partial C$ to have corners. The definition has some small modifications from the one defined in \cite{Spectral_networks}.
Fix a set of disjoint branch cuts $\{c_x\}_{x\in S}$. Let
$$U:=C-\bigsqcup_{x\in S}c_x$$
and write
$$p^{-1}(U):=\bigsqcup_{\alpha=1}^rU^{(\alpha)},$$
so that $p|_{U^{(\alpha)}}:U^{(\alpha)}\to U$ is a homeomorphism.

\begin{definition}\label{def:spectral_network}
    Let $p:L\to C$ be a branched covering map of orientable surfaces with simple branching. A \emph{spectral network $\cu{W}$ subordinate to $p:L\to C$} is a countable collection of oriented smooth embeddings $\{w:[0,1]\to C\}$, called \emph{walls}, together with a collection of points $\mf{s}\subset\partial C$, called the \emph{flag data}, satisfying the following conditions.
    \begin{enumerate}   
        \item \label{con:wall} For any wall $w$, the \emph{relative interior} $\Int(w):=w((0,1))$ satisfies $\Int(w)\subset U$ and all relative interiors of walls are disjoint. Moreover, every compact subset of $C-S$ intersects finitely many walls.%Every compact subset of $\ring{\Delta}-d\phi(S)$ intersects finitely many walls. In particular, the boundary $\partial\Delta$ only intersects finitely many walls.
        \item Each wall carries a label $\alpha\beta$, with $\alpha\neq\beta$ representing the sheets $U^{(\alpha)},U^{(\beta)}\subset L$.%, and reversing the orientation of a wall results in reserving the order of the label.
        \item \label{con:local_model} At each branch point $p$ for which the sheets $U^{(\alpha)},U^{(\beta)}$ comes together, there are three walls emitting out from $p$ as shown in Figure \ref{fig:branched_point}.
        \begin{figure}[H]
			\centering
			\includegraphics[width=30mm]{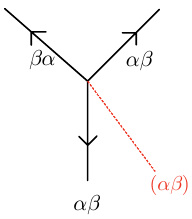}
			\caption{The local model around an $(\alpha\beta)$-branched point. The red dotted line represents the branch cut that we label by $(\alpha\beta)$ to remember the branched point is formed by gluing $U^{(\alpha)},U^{(\beta)}$.}
			\label{fig:branched_point}
		\end{figure}
        If a boundary point of a wall lies in $U$, there is a neighbourhood around it so that it looks like Figure \ref{fig:joint_1} or Figure \ref{fig:joint_2}
		\begin{figure}[H]
			\centering
			\includegraphics[width=40mm]{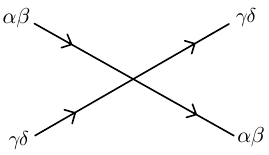}
			\caption{The joint where $\alpha\neq\delta,\beta\neq\gamma$.}
			\label{fig:joint_1}
		\end{figure}
		\begin{figure}[H]
			\centering
			\includegraphics[width=40mm]{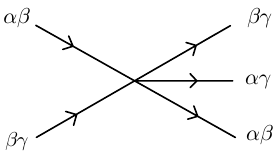}
			\caption{The joint where $\alpha\neq\beta,\beta\neq\gamma$ and $\alpha\neq\gamma$.}
			\label{fig:joint_2}
		\end{figure}
        \item \label{con:branch_change} If $w_1,w_2$ are walls so that $w_1(1)=w_2(0)$ lies on a branch cut of type $(\alpha\beta)$ and the label on $w_1$ involves $\alpha$ or $\beta$, then the label on $w_2$ is obtained from the one in $w_1$ by turning $\alpha$ to $\beta$ and $\beta$ to $\alpha$. %A wall $w$ is said to be an \emph{ancestor} of $w'$ if there exists a sequence of walls $w=w_1,w_2,\dots,w_k=w'$ such that $w_i(1)=w_{i+1}(0)$ and they either carry the same label or their labels are related by the above branch change. We also say $w'$ is a \emph{decedent} of $w$.
        \item \label{con:poset} For each $x\in\mf{s}$, there is neighbourhood $U_x\subset C$ and a partial ordering $<_x$ on the set of sheets $p^{-1}(U_x)$ of $U_x$ so that if $w$ is a wall with label $\alpha\beta$ and $w(1)=x$, then $\alpha<_x\beta$.
        %\item \label{local_finiteness} We call those boundary points in $d\phi(U)$ of an interior wall \emph{joints} and the set of joints is assumed to be locally finite on $\ring{\Delta}-d\phi(S)$.
        %\item \label{boundary_condition} If the boundary point $w(1)$ of a wall $w$ lies on $\partial\Delta$, we call it an \emph{endpoint}. A wall that contains an endpoint is called a \emph{boundary wall}.
    \end{enumerate}
    A wall that lies completely in $C-\partial C$ is called an \emph{interior wall} and we call those boundary points in $U$ of an interior wall \emph{joints}. The set of joints is denoted by $\mathrm{Joints}(\cu{W})$. If a boundary point of a wall lies on $\partial C$, we call it an \emph{endpoint}. A wall that contains an endpoint is called a \emph{boundary wall}. %Let $w$ be a wall without ancestors. Then set  
    %$$\mf{d}_{w(0)}:=w([0,1])\cup\bigcup_{w'\text{ is a descendent of w}}w'([0,1])$$
    %is called a \emph{ray} with based point $w(0)$.
    The set
    $$|\cu{W}|:=\bigcup_{w\in\cu{W}}w([0,1])$$
    is called the \emph{support} of
    $\cu{W}$. A spectral network subordinate to $p:L\to C$ is said to be \emph{non-degenerated} if each wall has at most one branch point. Otherwise, it is said to be \emph{degenerated}.
\end{definition}

%\begin{remark}
 %   The only differences between our definition and the one given in \cite{Spectral_networks} are
  %  \begin{enumerate}
  %      \item our spectral networks are not necessarily locally finite near the boundary $\partial C$,
   %     \item we allow $\partial C$ to have corners and the partial orderings $<_x$ in Condition \eqref{con:poset} are \emph{at} $x$ but not a neighbourhood of it.
   % \end{enumerate}
  %  The main reason for these modifications is that our spectral network is going to be defined on a polytope. Nevertheless, these mild modifications won't affect the non-abelianization process given in \cite{Spectral_networks} as long as $\cu{W}$ is locally finite on $C-\partial C$.
%\end{remark}

From now on, we assume $\Sigma$ is a smooth complete rational fan in $N_{\bb{R}}\cong\bb{R}^2$. Recall that the associated toric surface $X_{\Sigma}$ must be projective. Choose a strictly convex piecewise linear function $\phi^{\trop}:(N_{\bb{R}},\Sigma)\to\bb{R}$ and let
$$\Delta:=\{x\in M_{\bb{R}}:\inner{x,v_{\rho}}\leq\phi^{\trop}(v_{\rho}),\forall\rho\in\Sigma(1)\}\subset M_{\bb{R}}$$
be the corresponding polytope. By adding a linear function to $\varphi^{\trop}$ if necessary, we may assume $0$ is an interior lattice point of $\Delta$, and so $\phi^{\trop}(v_{\rho})>0$ for all $\rho\in\Sigma(1)$. We orient $\Delta$ in the anti-clockwise direction. For each cone $\tau\in\Sigma$, let $\check{\tau}\subset\Delta$ be the cell dual to $\tau$, that is
$$\check{\tau}:=\{x\in \Delta:\inner{x,v_{\rho}}=\phi^{\trop}(v_{\rho}),\forall \rho\subset\tau\}.$$
Define
$$\phi(\xi):=\frac{1}{2}\log\left(\sum_{m\in\Delta\cap M}e^{2\inner{m,\xi}}\right).$$
By abusing the notation, we denote the composition of $d\phi:N_{\bb{R}}\to T^*N_{\bb{R}}=N_{\bb{R}}\times M_{\bb{R}}$ with the projection $N_{\bb{R}}\times M_{\bb{R}}\to M_{\bb{R}}$ by $d\phi$. It is well-known that $d\phi$ identifies $N_{\bb{R}}$ and the interior $\ring{\Delta}$ of $\Delta$. We call $d\phi:N_{\bb{R}}\to\ring{\Delta}$ the \emph{Legendre transform}. %Let $\rho\in\Sigma(1)$ be a ray and $v_{\rho}$ its primitive generator. The curve  $t\mapsto d\phi(tv_{\rho})$ $,t\geq 0$ begins at the barycentre of $\Delta$ and intersects the barycentre of $\check{\rho}$ as $t\to\infty$. We denote by $\mathrm{Bar}(\partial\Delta)$ the set of cells in the first barycentric decomposition of $\partial\Delta$ and by $d\phi(\Sigma)$ the collection of image of cones in $\Sigma$ under $d\phi$. We also denote by $|d\phi(\Sigma(1))|\subset\ring{\Delta}$ the union of all the rays in $d\phi(\Sigma)$.
The function $\phi$ also gives a Hessian metric on $N_{\bb{R}}$:
$$g_{\phi}:=\sum_{i,j=1}^2\secpd{\phi}{\xi_i}{\xi_j}d\xi_i\otimes d\xi_j.$$
For a function $f$ defined on an open subset of $N_{\bb{R}}$, we write $\mathrm{grad}_{\phi}(f)$ the gradient of $f$ with respect to the metric $g_{\phi}$.

Let $\bb{L}^{\trop}$ be an $r$-fold tropical Lagrangian multi-section over $\Sigma$ and $\bb{L}$ be a $\Lambda_{\bb{L}^{\trop}}$-admissible (possibly immersed) Lagrangian multi-section with finitely many branch points. Let $\{c_{\xi}\}_{\xi\in S}$ be a set of disjoint branch cuts of $p:L\to N_{\bb{R}}$ so that
$$U:=N_{\bb{R}}-\bigsqcup_{\xi\in S}c_{\xi}$$
is contractible. Write
$$p_L^{-1}(U)=\bigsqcup_{\alpha=1}^rU^{(\alpha)},$$
where $U^{(\alpha)}\subset L$ are open subsets such that $p_L|_{U^{(\alpha)}}:U^{(\alpha)}\to U$ are diffeomorphisms. For each maximal cone $\sigma\in\Sigma(2)$, the open subset $\Int(\sigma)\cap U$ has a unique lift $U_{\sigma}^{(\alpha)}:=p_L^{-1}(\sigma)\cap U^{(\alpha)}$. The asymptotic condition $\bb{L}^{\infty}\subset\Lambda_{\bb{L}^{\trop}}^{\infty}$ then determines a lift $\sigma^{(\alpha)}\in\Sigma_{L^{\trop}}(2)$ of $\sigma$ for which
$$i(U_{\sigma}^{(\alpha)})^{\infty}\subset\{m(\sigma^{(\alpha)})\}\times(-\sigma^{\infty}).$$
Moreover, as $U$ is contractible, there exist smooth functions $\varphi^{(\alpha)}:U\to\bb{R}$ so that
$$i(U^{(\alpha)})=\{(d\varphi^{(\alpha)}(\xi),\xi)\in M_{\bb{R}}\times N_{\bb{R}}:\xi\in U\}.$$
We are interested in the gradient flow equations
\begin{equation}\label{eqn:grad_flow}
    \dot{\xi}(t)=\grad_{\phi}(\varphi^{(\alpha)}-\varphi^{(\beta)})(\xi(t)),
\end{equation}
for $\alpha,\beta=1,\dots,r$ being distinct. Let $I_{\xi}:=(0,a^+)$ be the interior the maximal domain of definition of a solution $\xi(t)$ to \eqref{eqn:grad_flow} with initial position $\xi(0^-)=\xi_0\in N_{\bb{R}}$. We are also interested in the Legendre transform $x(t):=d\phi(\xi(t))$ of their solutions, whose maximal domain of definition $I_x$ is of course $I_{\xi}$. It satisfies
\begin{equation}\label{eqn:LT_grad_flow}
    \dot{x}_i(t)=\pd{}{\xi_i}(\varphi^{(\alpha)}-\varphi^{(\beta)})(d\phi^{-1}((x(t)))),
\end{equation}
for $i=1,2$.
\begin{definition}\label{def:LT_line}
    We call a solution $x:[0,a^+)\to\mathring{\Delta}$ to \eqref{eqn:LT_grad_flow} an \emph{$\alpha\beta$-LT line} (LT stands for Legendre transform).
\end{definition}

%It is a general fact that if a smooth function $f$ is generic enough, no $f$-gradient lines are closed orbits or will spiral and accumulate. From now on, throughout the whole article, we make the following

%\begin{assumption}\label{ass:generic}
%    The potentials $\varphi^{(\alpha)}$'s are generic enough so that the gradient flow \eqref{eqn:grad_flow} has no closed orbits and the limits
 %   $$x(a^{\pm}):=\lim_{t\to a^{\pm}}x(t)$$
 %   exists in $\Delta$ for any LT line $x(t)$.
%\end{assumption}

Let $x:[0,a^+)\to\Delta$ be an $\alpha\beta$-LT line. If
$$x(a^+):=\lim_{t\to a^+}x(t).$$
exists in $\Delta$ and $x(a^+)\in\check{\rho}\subset\partial\Delta$ for some $\rho\in\Sigma(1)$, then by $\Lambda_{\bb{L}^{\trop}}$-admissibility, the limit
$$\lim_{t\to a^+}\dot{x}(t)$$
also exists and lies in the coset $m(\sigma^{(\alpha)})-m(\sigma^{(\beta)})+\rho^{\perp}\subset M_{\bb{R}}$, for any maximal cone $\sigma\in\Sigma$ containing $\rho$. We remark that this coset is independent of $\sigma$.

%By $\Lambda_{\bb{L}^{\trop}}$-admissibility, if $\xi(t)$ escape to infinity so that $x(t)$ hit the edge $\check{\rho}\subset\partial\Delta$ as $t\to a^+$, then the limit
%$$\dot{x}(a^+):=\lim_{t\to a^+}\dot{x}(t)$$
%exists and lies in $m(\sigma^{(\alpha)})-m(\sigma^{(\beta)})+\rho^{\perp}$ for some $\sigma\in\Sigma(2)$ contains $\rho$. After reparametrization, we can then keep propagating it along this direction after it hits $\check{\rho}$. Finally, for a vertex $\check{\sigma}=\{m\}\subset\Delta$, we can regard $m$ as a tangent vector on $M_{\bb{R}}$.

\begin{lemma}\label{lem:boundary_wall}
    Let $x(t)$ be an $\alpha\beta$-LT line so that $x(a^+)$ exists on $\partial\Delta$. Suppose $x(a^+)\in\check{\tau}$ for some $\tau\in\Sigma$. Then $a^+<\infty$ and
    $$m(\sigma^{(\alpha)})-m(\sigma^{(\beta)})\in\tau^{\vee}\cap M,$$
    for all maximal cone $\sigma$ containing $\tau$.
\end{lemma}
\begin{proof}
    Since $x(a^+)$ exists on $\partial\Delta\cap\check{\tau}$, $\dim(\tau)\geq 1$. Let $\rho\in\Sigma(1)$ be any ray contains in $\tau$. By assumption, the limit
    $$\lim_{t\to a^+}\inner{x(t),v_{\rho}}$$
    exists. If $a^+=\infty$, then
    $$\lim_{t\to a^+}\inner{\dot{x}(t),v_{\rho}}=0.$$
    By \eqref{eqn:LT_grad_flow} and $\Lambda_{\bb{L}^{\trop}}$-admissibility,
    $$\lim_{t\to a^+}\dot{x}(t)\in m(\sigma^{(\alpha)})-m(\sigma^{(\beta)})+\rho^{\perp},$$
    for any $\sigma\supset\rho$. This implies $(m(\sigma^{(\alpha)})-m(\sigma^{(\beta)}))|_{\rho}=0$, violating separability of $\bb{L}^{\trop}$. By elementary analysis, the left-handed derivative $\dot{x}(a^+)$ exists and $\dot{x}(t)\to\dot{x}(a^+)$ as $t\to a^+$.
    
    It remains to prove the desired slope condition. By definition, we have
    $$\inner{x(t),v_{\rho}}\leq\phi^{\trop}(v_{\rho}),$$
    for all $\rho\in\Sigma(1)$. If $x(a^+)\in\check{\tau}$, for some $\tau\in\Sigma$, then
    $$\inner{x(a^+),v_{\rho}}=\phi^{\trop}(v_{\rho}),$$
    for all rays $\rho\subset\tau$. Hence for any $t<a^+$,
    $$\inner{\frac{x(t)-x(a^+)}{t-a^+},v_{\rho}}=\frac{1}{t-a^+}\left(\inner{x(t),v_{\rho}}-\phi^{\trop}(v_{\rho})\right)\geq 0.$$
    for all $\rho\subset\tau$. This implies $\inner{\dot{x}(a^+),v_{\rho}}\geq 0$, for all $\rho\subset\tau$ and so $\dot{x}(a^+)\in\tau^{\vee}$. By \eqref{eqn:LT_grad_flow} and $\Lambda_{\bb{L}^{\trop}}$-admissibility again, we have
    $$\dot{x}(a^+)=\lim_{t\to a^+}\dot{x}(t)\in m(\sigma^{(\alpha)})-m(\sigma^{(\beta)})+\tau^{\perp},$$
    for any maximal cone $\sigma$ containing $\tau$. The result follows.
\end{proof}

We now specify a local model around each ramification point of a Lagrangian multi-section. We equip $\bb{C}^2\cong\bb{R}^2\oplus\sqrt{-1}\bb{R}^2\cong T^*\bb{R}^2$ with the standard symplectic form and denote by $p_2:\bb{C}^2\to\bb{C}$ the projection onto the second coordinate. The subset
$$L_2:=\{(z_1,z_2)\in\bb{C}^2:\ol{z}_1=z_2^2\}$$
is a 2-fold Lagrangian multi-section with respect to $p_2:\bb{C}^2\to\bb{C}$. We now make the following

\begin{definition}
    A Lagrangian multi-section $\bb{L}$ in $T^*M_{\bb{R}}$ is said to be have \emph{simple branching} if for each branch point $\xi_0\in S$, there exists a neighbourhood $U$ of $\xi_0$ such that
    $$p_L^{-1}(U)=U'\sqcup\bigsqcup_{\alpha=1}^{r-2}U^{(\alpha)},$$
    for some open sets $U',U^{(\alpha)}\subset L$ so that $p_L|_{U^{(\alpha)}}:U^{(\alpha)}\to U$ are diffeomorphisms and there exists an embedding $f:U\to\bb{C}$ mapping $\xi_0$ to $0\in\bb{C}$, a symplectic embedding $F:p_{N_{\bb{R}}}^{-1}(U)\to\bb{C}^2$ such that the diagram
    \begin{center}
    \begin{tikzcd}
    & p_{N_{\bb{R}}}^{-1}(U) \rar{F} \arrow{d}{p_{N_{\bb{R}}}} & \bb{C}^2 \arrow{d}{p_2} \\
    & U \rar{f} & \bb{C}
    \end{tikzcd}
    \end{center}
    commutes, and $(F\circ i)(U')=L_2\cap p_2^{-1}(f(U))$.
\end{definition}
By choosing a branch cut around a branch point, we can write the local model $L_2$ as
$$\left\{\left(d\left(\frac{2}{3}r^{\frac{3}{2}}\cos\left(\frac{3}{2}\theta\right)\right),re^{\sqrt{-1}\theta}\right)\in\bb{C}^2:r>0,\theta\in(-\varepsilon,2\pi-\varepsilon)\right\}$$
in terms of polar coordinate $(r,\theta)$ on $\bb{C}$. With respect to the flat metric on $\bb{C}$, it is a standard calculation \cite{Fukaya_asymptotic_analysis, Suen_TP2} that there are three solutions to the gradient flow
$$(\dot{r},\dot{\theta})=\grad\left(\frac{4}{3}r^{\frac{3}{2}}\cos\left(\frac{3}{2}\theta\right)\right)$$
that limit to the branch point $0\in\bb{C}$.
By our choice of branch cut, one can check that there are one incoming and two outgoing gradient lines and the branch cut lies between the two outgoing lines. This is exactly the local model around a branch point as shown in Figure \ref{fig:branched_point}.

Recall that our gradient flow depends on the Hessian metric $g_{\phi}$. However, the above flow line calculation is based on the flat metric. We therefore would like to modify the Hessian metric $g_{\phi}$ around branch points of $p_L:L\to N_{\bb{R}}$ so that this new metric is flat around the branch points and agrees with $g_{\phi}$ outside a large compact subset of $N_{\bb{R}}$. Such a metric can easily be constructed by the partition of unity argument. Denote such a modified metric by $\ol{g}_{\phi}$.

\begin{definition}
    The \emph{set of rays subordinate to $\bb{L}$} is the collection
    $$\cu{R}_{\bb{L}}:=\{\mf{d}:[0,a^+)\to N_{\bb{R}}\}$$
    of all embedded rays so that the following conditions are satisfied.
    \begin{enumerate}
        \item $\mf{d}|_{\mf{d}^{-1}(U)}$ is a maximal solution to the $\alpha\beta$-gradient flow equation
        $$\dot{\mf{d}}(t)=\grad_{\ol{g}_{\phi}}(\varphi^{(\alpha)}-\varphi^{(\beta)})(\mf{d}(t))$$
        for some distinct $\alpha,\beta=1,\dots,r$.
        \item If $\mf{d}:[0,a^+)\to N_{\bb{R}}$ satisfies some $\alpha\beta$-gradient gradient flow equation with $\mf{d}(0)\in S$, then $\mf{d}\in\cu{R}_{\bb{L}}$.
        \item If $\mf{d}_1,\mf{d}_2\in\cu{R}_{\bb{L}}$ satisfy the $\alpha\beta$- and $\beta\gamma$-gradient flow equation with $\mf{d}_1(t_1)=\mf{d}_2(t_2)$ for some $t_1\in\mf{d}_1|_{\mf{d}_1^{-1}(U)},t_2\in\mf{d}_2|_{\mf{d}_2^{-1}(U)}$, then there exists $\mf{d}\in\cu{R}_{\bb{L}}$ such that $\mf{d}$ satisfies the $\alpha\gamma$-gradient flow equation and $\mf{d}(0)=\mf{d}_1(t_1)=\mf{d}_2(t_2)$.
    \end{enumerate}
\end{definition}

\begin{assumption}
    We always assume $\bb{L}$ is generic enough so that all rays in $\cu{R}_{\bb{L}}$ intersect transversely.
\end{assumption}

We can apply Legendre transform to rays in $\cu{R}_{\bb{L}}$ and get a collection of LT lines on $\Delta$. They give a collection of walls $\cu{W}_{\bb{L}}$ in the most obvious way. Let
$$\mf{s}_{\lambda}:=\left\{x\in\partial\Delta_{\lambda}:\exists\mf{d}\in\cu{R}_{\bb{L}}\text{ such that }d\phi(\mf{d}(t))=x,\text{ for some }t\in(0,a^+)\right\}.$$
For $x\in\mf{s}_{\lambda}$, we declare $\alpha<_x\beta$ if and only if
\begin{equation}\label{eqn:slope_condition}
    m(\sigma^{(\alpha)})-m(\sigma^{(\beta)})\in\tau^{\vee}\cap M,
\end{equation}
for any maximal cone $\sigma\in\Sigma(2)$ contains $\rho$. By Lemma \ref{lem:boundary_wall}, $<_x$ is a well-defined partial ordering.

Let $\check{p}_L:L\to\ring{\Delta}$ be the composition $d\phi\circ p_L$, which is of course still a branched covering map of degree $r$. However, this map cannot be extended to a branched covering map over $\Delta$. To include the information along the boundary, we choose a $\lambda\in(0,1)$ and consider the shrunken polytope $\Delta_{\lambda}:=\lambda\Delta$. Assume $\lambda$ is close enough to 1 that $\check{p}_L^{-1}(\Delta_{\lambda})$ is isotopy to $L$. The restriction $\check{p}_L|_{\check{p}_L^{-1}(\Delta_{\lambda})}:\check{p}_L^{-1}(\Delta_{\lambda})\to\Delta_{\lambda}$ is then a branched covering map of degree $r$. We denote this covering map by $\check{p}_{L,\lambda}$.

\begin{definition}\label{def:spectral_network_L}
    If there exists $\lambda_0\in(0,1)$ so that $(\cu{W}_{\bb{L}},\mf{s}_{\lambda})$ is a spectral network subordinate to $\check{p}_{L,\lambda}:\check{p}_L^{-1}(\Delta_{\lambda})\to\Delta_{\lambda}$ for all $\lambda\in[\lambda_0,1)$, we call it the \emph{spectral network subordinate to $\bb{L}$}. As $\mf{s}_{\lambda}$ is determined by $\cu{W}_{\bb{L}}$, we simply say $\cu{W}_{\bb{L}}$ is a spectral network.
\end{definition}

\begin{remark}
    Such spectral network is necessarily finite because $\Delta$ has no punctures.
\end{remark}

     Because of the local models in Figure \ref{fig:branched_point}, \ref{fig:joint_1}, \ref{fig:joint_2} and Lemma \ref{lem:boundary_wall}, such spectral network may not exist for general $\Lambda_{\bb{L}^{\trop}}$-admissible $\bb{L}$. See Example \ref{eg:non_example}. We now provide a condition for which $\cu{W}_{\bb{L}}$ is a spectral network.

\begin{definition}\label{def:well_behaved}
We say $\bb{L}$ is \emph{well-behaved} if $\cu{R}_{\bb{L}}$ is finite and for any $\mf{d}\in\cu{R}_{\bb{L}}$ the limit
$$\lim_{t\to a^+}d\phi(\mf{d}(t))$$
exists on $\partial\Delta$.
\end{definition}

\begin{remark}\label{rmk:finite}
    In a conversation with Yoo Jae Nho, he demonstrated to the author that how exactness of $\bb{L}$ and boundedness of the primitive ensure the finiteness of the spectral network. In our case, an exact $\Lambda_{\bb{L}^{\trop}}$-admissible Lagrangian multi-section takes the form
    $$\bb{L}=\{((p^*)^{-1}(d\varphi(l))),-p(l))\in M_{\bb{R}}\times N_{\bb{R}}:l\in L\}$$
    for some smooth function $\varphi:L\to\bb{R}$ so that $\varphi$ is a smoothing of $\varphi^{\trop}$ near infinity. See \cite[Section 5.2]{OS}. The $\lambda_{M_{\bb{R}}}$-primitive of $\bb{L}$ is given by $i^*f+\varphi$, where $f:T^*M_{\bb{R}}\to\bb{R}$ is the function $f(x,\xi):=\inner{x,\xi}$. As $i^*f$ and $-\varphi$ satisfy the same asymptotic condition determined by $\bb{L}^{\infty}\subset\Lambda_{\bb{L}^{\trop}}^{\infty}$, we can choose a good enough smoothing $\varphi$ of $\varphi^{\trop}$ so that $i^*f+\varphi$ is bounded on $L$ and Nho's argument applied. However, as Nho's work is still in preparation, to respect his original idea, we include finiteness here as an assumption.
\end{remark}

Well-behavedness prohibits gradient lines from connecting two branched points and spinning around. Hence well-behavedness is in fact a generic assumption on $\bb{L}$.

\begin{theorem}\label{thm:existence}
    If $\bb{L}$ is a well-behaved $\Lambda_{\bb{L}^{\trop}}$-admissible Lagrangian multi-section, then $\cu{W}_{\bb{L}}$ is a non-degenerated spectral network.
\end{theorem}
\begin{proof}
All the properties of being a spectral network actually follows from the definition of $\cu{R}_{\bb{L}}$. We first check Condition \eqref{con:local_model} in Definition \ref{def:spectral_network_L}. In Section \ref{sec:spec_net_L}, we have already seen that the local model of $\bb{L}$ around ramification gives rise to three gradient lines as shown in Figure \ref{fig:branched_point}. By well-behavedness of $\bb{L}$, there are no gradient lines connecting two branch points. This implies the label of these three gradient lines around a branch point is exactly given by Figure \ref{fig:branched_point}. Suppose two gradient lines $\xi^{(\alpha\beta)},\xi^{(\beta\gamma)}$ with $\alpha\neq\gamma$, intersecting transversely at a point $\xi_0:=\xi^{(\alpha\beta)}(t_1)=\xi^{(\beta\gamma)}(t_2)\in U$, for some time $t_1,t_2$. Then the gradient line $\xi^{(\alpha\gamma)}$ of
$$\dot{\xi}(t)=\grad_{\ol{g}_{\phi}}(\varphi^{(\alpha)}-\varphi^{(\gamma)})(\xi(t))$$
hat pass through $\xi_0$ has tangent vector $$\dot{\xi}^{(\alpha\beta)}(t_1)+\dot{\xi}^{(\beta\gamma)}(t_2)$$
at $\xi_0$. This gives the bottom local models in Figure \ref{fig:joint_1} or Figure \ref{fig:joint_2}. When a gradient line hits a branch cut, the label will change according to Condition \eqref{con:branch_change}, simply by topological reason.

Now, take the Legendre transform to obtain a collection of walls on $\Delta$. By well-behavedness, LT-lines can only terminate on $\partial\Delta$, so it cannot connect two branched points. Hence $\cu{W}_{\bb{L}}$ must be non-degenerated.
\end{proof}

We would like to point out that a generic embedded 2-fold Lagrangian multi-section or an immersed one with index 1 immersed double points is well-behaved, that is, the function $\varphi^{(1)}-\varphi^{(2)}$ has at most index 1 critical point in $U$. First, $\cu{R}_{\bb{L}}$ is finite (even without exactness) as we just have two labels, namely 12 or 21. Note that critical points of $\varphi^{(\alpha)}-\varphi^{(\beta)}$ on $U$ corresponds to immersed points of $\bb{L}$. As $\bb{L}$ is assumed to have index 1 immersed double points, critical points of $\varphi^{(\alpha)}-\varphi^{(\beta)}$ on $U$ are isolated and have index 1. This implies the Hessian of $\varphi^{(1)}-\varphi^{(2)}$ at critical points is non-degenerated and has only one negative eigenvalue, meaning that they are all saddle points and elementary calculus tells us that at such critical points, a wall can keep propagating after hitting the critical point. Geometrically, this corresponds to a collision of two $(12)$-type branch points and leads the local model in Figure \ref{fig:joint_1} (with $\gamma\delta=\alpha\beta$). Floer theoretically, this corresponds to the fact that a 2-dimensional non-compact immersed Lagrangian whose immersed sector is concentrated at degree 1 is (tautologically) unobstructed. Moreover, as the local models of these Lagrangian multi-sections are given by hyper-K\"ahler rotation of hyper-elliptic curves, they admit non-degenerated spectral networks if the hyper-elliptic curve is generic\footnote{The hyper-elliptic curve in \cite[Section 5.2]{OS} is chosen to be real. But we can relax the construction by multiplying the polynomial $f_d$ by a generic $e^{\sqrt{-1}\theta}$ and choosing $c_{\frac{d}{2}+1}=0$ when $d$ is even. The resulting spectral network will then be non-degenerated.}. This provides many examples of 2-fold Lagrangian multi-sections that admit non-degenerated spectral networks.

\begin{remark}
    The collection $\cu{R}_{\bb{L}}$ gives a collection of gradient trees, which we think of as Maslov zero holomorphic disks bounded by $\bb{L}$ and fibers of $p_{N_{\bb{R}}}:T^*M_{\bb{R}}\to N_{\bb{R}}$. See \cite{FO, Fukaya_asymptotic_analysis, Abouzaid09} for more details. Therefore, we expect the existence of such a spectral network should be related to the Floer-theoretic obstruction of $\bb{L}$.
\end{remark}

We now introduce solitons. They are the key ingredients in constructing the non-abelianization map.

\begin{definition}\label{def:soliton}
    A \emph{soliton} associated to a wall $w\in\cu{W}$ with label $\alpha\beta$ is pair $(s_{\alpha\beta},x)$, where $s_{\alpha\beta}:[0,1]\to L$ a smooth path and $x\in\Int(w)$, satisfying the following conditions.
    \begin{enumerate}
        \item $s_{\alpha\beta}(0)=x^{(\alpha)}\in U^{(\alpha)},s_{\alpha\beta}(1)=x^{(\beta)}\in U^{(\beta)}$ with $\check{p}_L(x^{(\alpha)})=\check{p}_L(x^{(\beta)})=x$.
        \item For each joint $\mf{j}$, there exists a small ball $B_{\varepsilon}(\mf{j})\subset C$ centred at $\mf{j}$ such that $$(\check{p}_L\circ s_{\alpha\beta})([0,1])-\bigcup_{\mf{j}}B_{\varepsilon}(\mf{j})\subset|\cu{W}|-\bigcup_{\mf{j}}B_{\varepsilon}(\mf{j}).$$
        \item The tangent vector $p_*(\dot{s}_{\alpha\beta}(0))$ is opposite to the orientation of $w$ while $p_*(\dot{s}_{\alpha\beta}(1))$ agrees.
    \end{enumerate}
    Two solitons $(s_{\alpha\beta},x),(s_{\alpha\beta}',x')$ are said to be \emph{equivalent} if they are homotopic in $L$. Denote by $\cu{S}(w,\alpha\beta)$ the set of solitons associated to a wall $w$ with label $\alpha\beta$ modulo equivalence.
\end{definition}

Solitons are required to satisfy some ``traffic rules". Roughly speaking, these traffic rules ensure the so-called \emph{soliton contents}, which is a collection $\{(s_{\alpha\beta},x)\}$ of solitons associated to each wall, is ``locally constant" in $x$. It also tells us how solitons are added when scattering occurs as in Figure \ref{fig:joint_2}. We refer to \cite[Section 9.3, \textbf{ST1-ST3}]{Spectral_networks} for details. In our case, solitons are lift of gradient flow lines on $N_{\bb{R}}$. When $\cu{W}_{\bb{L}}$ is a spectral network, the set of solitons over a wall is always locally finite and the way to determine the new solitons from the scattered wall is given in \cite[\textbf{ST3}]{Spectral_networks}.

%Given any path $\gamma:[0,1]\to L$, we define its \emph{symplectic energy} to be
%$$E(\gamma):=\int_0^1\gamma^*i^*\lambda,$$
%where $\lambda=\sum_i\xi_idx_i$ is the Liouville 1- form on $T^*M_{\bb{R}}$. By the Lagrangian condition, it is obvious that $E(\gamma)$ on the homotopy class of $\gamma$ relative to the boundaries $\gamma(0),\gamma(1)$. In particular, this holds for any soliton $(s_{\alpha\beta},x)$.

\begin{remark}\label{rmk:grad_trees}
    As we think of gradient trees as the tropical limit of Maslov index zero disks, we should also think of solitons as (part of) the boundary of Maslov index zero $J_{\ol{g}_{\phi}}$-holomorphic disks bounded by $\bb{L}$ and a fiber. See \cite{Fukaya_asymptotic_analysis} for the origin of this interpretation. 
\end{remark}

\subsection{Non-abelianization}\label{sec:non_abel}

Let $C$ be an orientable surface with or without boundary/punctures and $p:L\to C$ be an $r$-fold covering map with branch locus $S\subset C-\partial C$. Suppose there exists a non-degenerated spectral network $\cu{W}$ subordinate to $p:L\to C$. We recall the construction of the non-abelianization of a rank 1 twisted-flat $\Bbbk^{\times}$-local system.

Let $X$ be a smooth orientable surface and $\til{X}:=TX-0_X$ be the tangent bundle of $X$ with the zero section $0_X$ removed. We denote by $\pi_X:\til{X}\to X$ the projection map. For any smooth path $\gamma:I\to X$ with non-vanishing speed, there exists a canonical lift $\til{\gamma}:I\to\til{X}$ given by
$$\til{\gamma}(t):=(\gamma(t),\dot{\gamma}(t)).$$
Hence the spectral network $\cu{W}$ can be lifted to $\til{C}$ naturally. We denote this lifting by $\til{\cu{W}}$.

\begin{definition}
A rank $r$ local system on $\til{X}$ is said to be \emph{twisted-flat} if it has holonomy $-\Id$ around the fiber class $S_X\in H_1(\til{X};\bb{Z})$. Denote the moduli space of twisted-flat rank $r$ local systems on $\til{X}$ by $\cu{M}^{\tw}(\til{X},GL_r)$. 
\end{definition}

By choosing a spin structure on $X$, there is a well-known correspondence
$$\cu{M}^{\tw}(\til{X},GL_r)\xrightarrow{\sim}\cu{M}(X,GL_r),$$
where the latter is the moduli of rank $r$ local systems on $X$. The spin determines a fiberwise double cover $\pi:\til{X}'\to\til{X}$. Given $\cu{E}^{\tw}\in\cu{M}^{\tw}(\til{X},GL_r)$, the pull-back $\pi^*\cu{E}^{\tw}$ has holonomy $\Id$ along the fiber class $S_X'\in H_1(\til{X};\bb{Z})$ and hence descends to a local system on $X$. The other direction is given by mapping $\cu{E}\mapsto\pi^*\cu{E}\otimes\cu{O}^{\tw}$, where $\cu{O}^{\tw}\in\cu{M}^{\tw}(\til{X},GL_1)$ is the rank 1 twisted-flat local system that has trivial monodromy on $H_1(X;\bb{Z})\subset H_1(\til{X};\bb{Z})$ (via $\gamma\mapsto\til{\gamma}$).

Let $\cu{L}^{\tw}\in\cu{M}^{\tw}(\til{L},GL_1)$. We now construct the non-abelianization $\Psi_{\cu{W}}(\cu{L}^{\tw})$. Choose a set of branch cuts $\{c_x\}_{x\in S}$. Put
$$B:=C-|\cu{W}|-\bigsqcup_{x\in S}c_x.$$
For $x\in B$ and any lift $\til{x}\in\til{U}$, we define
$$E_{\til{x}}:=\bigoplus_{\alpha=1}^r\cu{L}^{\tw}_{\til{x}^{(\alpha)}}.$$

We have three cases to consider.
\begin{enumerate}
    \item [(I)] Suppose two points $x_1,x_2$ lie in the same connected component of $B$. Let $\gamma\subset B$ be a smooth embedded path from $x_1$ to $x_2$ inside the connected component. We then identify $E_{\til{x}_1}$ and $E_{\til{x}_2}$ by
    $$\ol{P}_{\til{x}_1\to\til{x}_2}:=\left(\ol{P}_{\til{x}_1^{(\alpha)}\to\til{x}_2^{(\alpha)}}\right)\in\Hom(E_{\til{x}_1},E_{\til{x}_2}),$$
    where $\ol{P}_{\til{x}_1^{(\alpha)}\to\til{x}_2^{(\alpha)}}\in\Hom(\cu{L}_{\til{x}_1^{(\alpha)}},\cu{L}_{\til{x}_2^{(\alpha)}})$ is the parallel transport of $p_*\cu{L}^{\tw}$ along the lift $\til{\gamma}^{(\alpha)}$.
    \item [(II)] Suppose two points $x_1,x_2\in B$ are separated by a branch cut so that the $\alpha$-sheet changes to the $\alpha'$-sheet. Let $\gamma\subset C$ be a smooth embedded path that crosses the branch cut once from $x_1$ to $x_2$. It determines the lifts $\til{x}_1,\til{x}_2\in\til{U}$. We identify $E_{\til{x}_1}$ and $E_{\til{x}_2}$ by
    $$\ol{P}_{\til{x}_1\to\til{x}_2}:=\left(\ol{P}_{\til{x}_1^{(\alpha)}\to\til{x}_2^{(\alpha')}}\right)\in\Hom(E_{\til{x}_1},E_{\til{x}_2}),$$
    where $\gamma^{(\alpha\alpha')}$ is the lift of $\gamma$ so that $\gamma^{(\alpha\alpha')}(0)=x^{(\alpha)},\gamma^{(\alpha\alpha')}(1)=x^{(\alpha')}$ and $\ol{P}_{\til{x}_1^{(\alpha)}\to\til{x}_2^{(\alpha')}}$ along $\til{\gamma}^{(\alpha\alpha')}$.
    
    \item [(III)] Let $x_1,x_2\in V$ be separated by a wall $w$ with label $\alpha\beta$ and $\gamma\subset C$ be a smooth embedded path going from $x_1$ to $x_2$, intersecting $\Int(w)$ transversely at a point $x$. The points $x,x_1,x_2$ have lifts $\til{x},\til{x}_1,\til{x}_2\in\til{U}$ determined by the path $\gamma$ and its tangent direction. Denote by $\ol{P}_{\til{x}_1\to\til{x}},\ol{P}_{\til{x}\to\til{x}_2}$ the parallel transports of $\cu{L}^{\tw}$ induced by the two obvious segments of $\til{\gamma}$. Suppose $(s_{\alpha\beta},x)\in\cu{S}(w,\alpha\beta)$ is a soliton associated to $w$. We can lift $(s_{\alpha\beta},x)$ canonically to a path $\til{s}_{\alpha\beta}$ in $\til{L}$, going from $\til{x}_s^{(\alpha)}:=(x^{(\alpha)},\dot{s}_{\alpha\beta}(0))\in\til{U}^{(\alpha)}$ to $\til{x}_s^{(\beta)}:=(x^{(\beta)},\dot{s}_{\alpha\beta}(1))\in\til{U}^{(\beta)}$. However, $\til{x}_s^{(\alpha)},\til{x}_s^{(\beta)}$ may not be the same as the lifts $\til{x}^{(\alpha)}\in\til{U}^{(\alpha)},\til{x}^{(\beta)}\in\til{U}^{(\beta)}$. We choose two paths $\til{A}_{\pm}\subset\pi_C^{-1}(x)$ with $\til{A}_+$ going from $(x,\dot{w}_x)$ to $\til{x}$ and $\til{A}_-$ going from $\til{x}$ to $(x,-\dot{w}_x)$ so that their concatenation $\til{A}_-*\til{A}_+$ is homotopic (with boundary points fixed) to a path that contains $\til{x}$ with zero winding number around the origin. We then lift them canonically to paths $\til{A}_{\pm}^{(\alpha)}\subset\pi_L^{-1}(x^{(\alpha)}),\til{A}_{\pm}^{(\beta)}\subset\pi_L^{-1}(x^{(\beta)})$. Put
    $$\til{a}(s_{\alpha\beta}):=(-\til{A}_+^{(\beta)})*\til{s}_{\alpha\beta}*\til{A}_-^{(\alpha)},$$
    which is a path from $\til{x}^{(\alpha)}$ to $\til{x}^{(\beta)}$. Let
    $\ol{P}(\til{a}(s_{\alpha\beta}),\til{x})$ be the parallel transport of $\cu{L}^{\tw}$ along $\til{a}(s_{\alpha\beta})$ and define the \emph{wall-crossing automorphism associated to $w$}
    $$\ol{\Theta}_{w,\til{x}_1\to\til{x}_2}:=\ol{P}_{\til{x}\to\til{x}_2}\left(\Id_{\til{x}}+\sum_{(s_{\alpha\beta},x)\in\cu{S}(w,\alpha\beta)}\ol{P}(\til{a}(s_{\alpha\beta}),\til{x})\right)\ol{P}_{\til{x}_1\to\til{x}}\in\Hom(E_{\til{x}_1},E_{\til{x}_2}).$$
\end{enumerate}

Let $\gamma:[0,1]\to C-S-\mathrm{Joints}(\cu{W})$ be a generic path that intersects $|\cu{W}|$ transversely. By local finiteness of $\cu{W}$, it intersects finite may walls. Let $\ol{\Theta}_{\cu{W},\gamma}$ be the path-ordered product induced by parallel transport and $\ol{\Theta}_w$'s along $\gamma$.

\begin{theorem}[\cite{Spectral_networks}]\label{thm:nonabel}
    For any contractible loop $\gamma:[0,1]\to C$ that avoids $S\cup\mathrm{Joints}(\cu{W})$, we have $\ol{\Theta}_{\cu{W},\gamma}=\Id_{E_{\til{\gamma}(0)}}$.
\end{theorem}

Theorem \ref{thm:nonabel} gives a twisted-flat rank $r$ local system $\ol{\Psi}_{\cu{W}}(\cu{L}^{\tw})$ on $\til{C}$, which we call the \emph{non-abelianization of $\cu{L}^{\tw}$ with respect to $\cu{W}$}. Via the identifications $\cu{M}^{\tw}(\til{L},GL_1)\cong\cu{M}(\til{L},GL_1)$ and $\cu{M}^{\tw}(\til{C},GL_r)\cong\cu{M}(\til{C},GL_r)$, we also say $\ol{\Psi}_{\cu{W}}(\cu{L})$ is the \emph{non-abelianization of $\cu{L}$ with respect to $\cu{W}$}. As pointed out in \cite[Section 10.2]{Spectral_networks}, the non-abelianization $\ol{\Psi}_{\cu{W}}(\cu{L}))$ is more than just a local system on $C$. It carries the flag data given by the partial ordering $<_x$ at each point $x\in\mf{s}$, namely, there is a flag of subspaces in the fibers $\ol{\Psi}_{\cu{W}}(\cu{L})_x$, $x\in\mf{s}$, for which they are invariant under the flat connection of $\ol{\Psi}_{\cu{W}}(\cu{L}))$.

%So far, we haven't seen the effect of the Novikov parameter. This is because we have assumed $\cu{W}$ to be locally finite. Suppose we are given a general spectral network $\cu{W}$. In this case, a path can intersect infinitely many walls, so path-ordered products are not defined. Here is where the Novikov parameter $T$ comes to the rescue. Let $K\subset C-S-\partial C$ be a compact subset and define
%$$W_K:=K\cap|\cu{W}|.$$

%\begin{lemma}\label{lem:finite}
 %   $W_K$ intersects finitely many walls of $\cu{W}$.
%\end{lemma}

%By Lemma \ref{lem:finite}, there is a compact submanifold with boundary $C_{\lambda}\subset C$ and a spectral network $\cu{W}^{\leq\lambda}$ subordinate to the covering map $p|_{p^{-1}(C_{\lambda})}:p^{-1}(C_{\lambda})\to C_{\lambda}$ whose support is given by $W_{\lambda}$ and we call $\cu{W}^{\leq\lambda}$ the \emph{$\lambda$-truncation of $\cu{W}$}. We are now safe to apply non-abelianization to $\cu{L}^{\tw}$ to produce a twisted-flat rank $r$ local system on $\til{C}_{\lambda}$ defined over $\bb{K}/(T^{\lambda})$. By letting $\lambda\to\infty$, we get a twisted-flat rank $r$ local system $\ol{\Psi}_{\cu{W}}(\cu{L}^{\tw})$ on $\til{C}$, defined over $\bb{K}$. We refer $\ol{\Psi}_{\cu{W}}(\cu{L}^{\tw})$ (resp. $\ol{\Psi}_{\cu{W}}(\cu{L})$) as the non-abelianization of $\cu{L}^{\tw}$ (resp. $\cu{L}$) with respect to the spectral network $\cu{W}$.

\subsection{Main construction}\label{sec:mirror_construction}

Let $\bb{L}^{\trop}$ be a tropical Lagrangian multi-section over $\Sigma$. Let $\cu{W}_{\bb{L}^{\trop}}$ be a non-degenerated spectral subordinated to some covering map $p:L\to\Delta$ (not necessarily admit a Lagrangian immersion $L\to T^*M_{\bb{R}}$) over with flag data given by \eqref{eqn:slope_condition}. Let $\cu{L}^{\tw}$ be a rank 1 twisted-flat $\Bbbk^{\times}$-local system on the $S^1$-bundle $\til{L}$. We now use the non-abelianization of $\cu{L}^{\tw}$ with respect to $\cu{W}$ to construct a rank $r$ toric vector bundle on $X_{\Sigma}$.

Our idea is to construct a 1-cocycle via path-ordered product along the boundary $\partial\Delta_{\lambda}$ of the $\lambda$-shrunken polytope $\Delta_{\lambda}$. By the finiteness of $\cu{W}_{\bb{L}^{\trop}}$ on the boundary $\partial\Delta$, we can choose $\lambda\sim 1$ such that
\begin{itemize}
    \item $\Delta-\ring{\Delta}_{\lambda}$ contains no joints.
    \item If a wall $w$ intersects $\Int(\lambda\check{\tau})\subset\partial\Delta_{\lambda}$, for some $\tau\in\Sigma$, we continue to have
    $$m(\sigma^{(\alpha)})-m(\sigma^{(\beta)})\in\tau^{\vee}\cap M,$$
    for all maximal cone $\sigma\supset\tau$.
\end{itemize}

\begin{remark}
    For instance, if $w$ is a wall with $w(1)\in\check{\sigma}$, shrinking $\Delta$ by $\lambda$ may make $w$ intersects the relative interior of one of the adjacent edge $\check{\rho}\supset\check{\sigma}$. Nevertheless, we have
    $$m(\sigma^{(\alpha)})-m(\sigma^{(\beta)})\in\sigma^{\vee}\cap M\subset\rho^{\vee}\cap M.$$
    Hence holomorphicity is still preserved.
\end{remark}

%By the finiteness of $\cu{W}_{\bb{L}}$ near the boundary $\partial\Delta$ (Condition \eqref{con:finite_walls}), there exists a compact subset $K\subset\ring{\Delta}$ that contains all the branched points of $\check{p}_L$ and joints of $\cu{W}_{\bb{L}}$.
%For each vertex $\check{\sigma}\subset\Delta$, let $V_{\sigma}\subset\Delta$ be a small neighbourhood of $\check{\sigma}$ so that $V_{\sigma}$ is disjoint from 

%$$|\cu{W}_{\bb{L}}|\cup\ol{|d\phi(\Sigma(1))|\cup\bigsqcup_{\xi\in S}d\phi(c_{\xi})}.$$
On each edge $\check{\rho}\subset\partial\Delta$, fix a point $x_{\rho}\in\Int(\check{\rho})-|\cu{W}_{\bb{L}^{\trop}}|$. This point tells us when we should change from the frame on $U(\sigma_1)$ to $U(\sigma_2)$, where $\sigma_1\cap\sigma_2=\rho$. See Case (III) below. We also make the following notation conventions.

\begin{notation}\label{conv:edge}
    Let $x\in\partial\Delta-\bigcup_{\rho\in\Sigma(1)}\{x_{\rho}\}$.
    \begin{enumerate}
        \item If $x$ is a vertex of $\Delta$, denote by $\sigma_x$ the corresponding maximal cone.
        \item If $x\in\Int(\check{\rho}_x)$ for some $\rho_x\in\Sigma(1)$, denote by $\sigma_x$ the maximal cone for which $\{x\},\check{\sigma}_x$ lies in the same component of $\check{\rho}_x-\{x_{\rho_x}\}$.
    \end{enumerate}
    When $x=w(1)$, for some boundary $w$, we write $\sigma_x,\rho_x$ as $\sigma_w,\rho_w$.
\end{notation}

As $\cu{W}_{\bb{L}^{\trop}}$ is non-degenerated, all LT lines (after reversing the label if necessary) end on $\partial\Delta$. We define
$$B_{\lambda}:=\partial\Delta_{\lambda}-|\cu{W}_{\bb{L}^{\trop}}|-\bigsqcup_{\rho\in\Sigma(1)}\{\lambda x_{\rho}\}-\bigsqcup_{\xi\in S}d\phi(c_{\xi}).$$
For each point $\til{x}\in\til{B}_{\lambda}$, we put
$$E_{\til{x}}:=\bigoplus_{\alpha=1}^r\cu{L}_{\til{x}^{(\alpha)}}^{\tw}$$
as before. This time we have four cases to consider.

\begin{enumerate}
    \item [(I)] Suppose two points $x_1,x_2$ lie in the same connected component of $B_{\lambda}$. Then $\sigma_{x_1}=\sigma_{x_2}=:\sigma$. Let $\gamma\subset B_{\lambda}$ be an embedded path from $x_1$ to $x_2$. We then identify $E_{\til{x}_1}\otimes_{\Bbbk}\Bbbk[\sigma^{\vee}\cap M]$ and $E_{\til{x}_2}\otimes_{\Bbbk}\Bbbk[\sigma^{\vee}\cap M]$ by the parallel transport $$P_{\til{x}_1\to\til{x}_2}:=\left(\ol{P}_{\til{x}_1^{(\alpha)}\to\til{x}_2^{(\alpha)}}\otimes 1\right)\in\Hom(E_{\til{x}_1},E_{\til{x}_2})\otimes_{\Bbbk}\Bbbk[\sigma^{\vee}\cap M].$$
    \item [(II)] Suppose two points $x_1,x_2\in B_{\lambda}$ are separated by a branch cut of type $(\alpha\beta)$. We continue to have $\sigma_{x_1}=\sigma_{x_2}=:\sigma$. Let $\gamma\subset \partial\Delta_{\lambda}$ be an embedded path from $x_1$ to $x_2$, passing through the branch cut once. We then identify $E_{\til{x}_1}\otimes_{\Bbbk}\Bbbk[\sigma^{\vee}\cap M]$ and $E_{\til{x}_2}\otimes_{\Bbbk}\Bbbk[\sigma^{\vee}\cap M]$ by the parallel transport $$P_{\til{x}_1\to\til{x}_2}:=\left(\ol{P}_{\til{x}_1^{(\alpha)}\to\til{x}_2^{(\alpha')}}\otimes 1\right)\in\Hom(E_{\til{x}_1},E_{\til{x}_2})\otimes_{\Bbbk}\Bbbk[\sigma^{\vee}\cap M].$$
    \item [(III)] Suppose $x_1,x_2\in B_{\lambda}$ are separated by the point $\lambda x_{\rho}$ for some $\rho\in\Sigma(1)$. Then $\sigma_{x_1}\cap\sigma_{x_2}=\rho$. Let $\gamma\subset \partial\Delta_{\lambda}$ be an embedded path from $x_1$ to $x_2$ and pass through $\lambda x_{\rho}$ once. Define the diagonal matrix
    $$G_{\sigma_{x_1}\sigma_{x_2}}^{\mathrm{sf}}:=\left(\ol{P}_{\til{x}_1^{(\alpha)}\to \til{x}_2^{(\alpha)}}\otimes z^{m(\sigma_{x_1}^{(\alpha)})-m(\sigma_{x_2}^{(\alpha)})}\right)\in\Hom(E_{\til{x}_1},E_{\til{x}_2})\otimes_{\Bbbk}\Bbbk[\rho^{\vee}\cap M],$$
    where $\sigma_{x_1}^{(\alpha)},\sigma_{x_2}^{(\alpha)}$ are determined by
    $$i(p_L^{-1}(\sigma_{x_i})\cap U^{(\alpha)})^{\infty}\subset\{m(\sigma_{x_i}^{(\alpha)})\}\times(-\sigma_{x_i})^{\infty},$$
    for $i=1,2$. The superscript ``$\mathrm{sf}$" stands for ``\emph{semi-flat}" as $G_{\sigma_{x_1}\sigma_{x_2}}^{\mathrm{sf}}$ comes from the ``classical" part of $\bb{L}$. See \cite{LYZ, Fukaya_asymptotic_analysis} for the exposition of these terminologies.
    \item [(IV)] Now, suppose two chosen points $x_1,x_2\in B_{\lambda}$ are separated by some boundary wall $w\in\cu{W}_{\bb{L}^{\trop}}$ with label $\alpha\beta$ intersecting $\Int(\lambda\check{\tau})$ for some $\tau\in\Sigma$. Again, let $\gamma\subset \partial\Delta_{\lambda}$ be an embedded path from $x_1$ to $x_2$ and only intersects $w$ at a point $x\in\Int(w)$. Using the notation convention in Notation \eqref{conv:edge}, we define
    $$\Theta_{w,\til{x}_1\to\til{x}_2}:=P_{\til{x}\to\til{x}_2}\left(\Id_{E_{\til{x}}}+\sum_{(s_{\alpha\beta},x)\in\cu{S}(w,\alpha\beta)}P(\til{a}(s_{\alpha\beta}),\til{x})\otimes z^{m(\sigma_w^{(\alpha)})-m(\sigma_w^{(\beta)})}\right)\ol{P}_{\til{x}_1\to\til{x}}.$$
    Then by the flag data \eqref{eqn:slope_condition}, we have $\Theta_{w,\til{x}_1\to\til{x}_2}\in\Hom(E_{\til{x}_1},E_{\til{x}_2})\otimes_{\Bbbk}\Bbbk[\tau^{\vee}\cap M]$.
    %\begin{enumerate}
     %   \item If $w$ hits $\Int(\check{\rho}_w)$, we have $$\Theta_{w,\til{x}_1\to\til{x}_2}\in\Hom(E_{\til{x}_1},E_{\til{x}_2})\otimes_{\bb{K}}\bb{K}[\rho_w^{\vee}\cap M],$$
      %  by Case (1) in Lemma \ref{lem:boundary_wall}.
       % \item If $w$ hits a vertex $\check{\sigma}_w=\{m_w\}$ and $\dot{w}_{ext}(1)\neq m_w$, we have $$\Theta_{w,\til{x}_1\to\til{x}_2}\in\Hom(E_{\til{x}_1},E_{\til{x}_2})\otimes_{\bb{K}}\bb{K}[\rho_w^{\vee}\cap M],$$
      %  by Case (2) in Lemma \ref{lem:boundary_wall}.
      %  \item If $w$ hits a vertex $\check{\sigma}_w=\{m_w\}$ and $\dot{w}_{ext}(1)=m_w$, we have $$\Theta_{w,\til{x}_1\to\til{x}_2}\in\Hom(E_{\til{x}_1},E_{\til{x}_2})\otimes_{\bb{K}}\bb{K}[\sigma_w^{\vee}\cap M],$$
      %  by Case (3) in Lemma \ref{lem:boundary_wall}.
%\end{enumerate}
\end{enumerate}

%Suppose the lifts $x_1^{(\alpha)}\in\Int(\sigma_1^{(\alpha)}),x_2^{(\alpha')}\in\Int(\sigma_2^{(\alpha')})$ are separated by a lift of $\rho-K$. Let $\gamma_{\alpha\alpha'}:[0,1]\to L$ be a short path joining $x_1^{(\alpha)}\in\Int(\sigma_1^{(\alpha)}),x_2^{(\alpha')}\in\Int(\sigma_2^{(\alpha')})$ through a lift of $\rho-K$. Define
%$$G_{\rho}^{\mathrm{sf}}:=P_{\sigma_1^{(\alpha)}\sigma_2^{(\alpha')}}z^{m(\sigma_2^{(\alpha')})-m(\sigma_1^{(\alpha)})},$$
%where $P_{\sigma_1^{(\alpha)}\sigma_2^{(\alpha')}}$ denote the parallel transport along the short path $\gamma_{\alpha\alpha'}$.

We consider the path-ordered product $\Theta_{\cu{W}_{\bb{L}^{\trop}},\partial\Delta_{\lambda}}$ induced by parallel transports, $G_{\sigma_1\sigma_2}^{\mathrm{sf}}$'s and $\Theta_w$'s along the loop $\partial\Delta_{\lambda}$, which is contractible in $\Delta$. We choose a parametrization $\gamma:[0,1]\to\partial\Delta_{\lambda}$ so that $\gamma(0)=\gamma(1)\in\check{\sigma}$ for some $\sigma\in\Sigma(2)$.

\begin{proposition}\label{prop:consistent}
We have $\Theta_{\cu{W}_{\bb{L}^{\trop}},\partial\Delta_{\lambda}}=\Id_{E_{\gamma(0)}}$.
\end{proposition}
\begin{proof}
    By the obvious relations
    \begin{align*}
        P_{\til{x}_2\to\til{x}_3}P_{\til{x}_1\to\til{x}_2}=&\,P_{\til{x}_1\to\til{x}_3},\\
        z^{m(\sigma_1^{(\alpha)})-m(\sigma_2^{(\beta)})}z^{m(\sigma_2^{(\beta)})-m(\sigma_3^{(\gamma)})}=&\,z^{m(\sigma_1^{(\alpha)})-m(\sigma_3^{(\gamma)})},
    \end{align*}
    and Theorem \ref{thm:nonabel}, we have
    $$\Theta_{\cu{W}_{\bb{L}^{\trop}},\partial\Delta_{\lambda}}^{(\alpha\beta)}=\ol{\Theta}_{\cu{W}_{\bb{L}^{\trop}},\gamma}^{(\alpha\beta)}\otimes z^{m(\sigma^{(\alpha)})-m(\sigma^{(\beta)})}=\delta^{(\alpha\beta)}\otimes z^{m(\sigma^{(\alpha)})-m(\sigma^{(\beta)})}=\delta^{(\alpha\beta)}\otimes 1,$$
    where $\delta^{(\alpha\beta)}$ is the Kronecker delta. This proves the theorem.
\end{proof}

\begin{remark}
    It seems that interior walls of $\cu{W}_{\bb{L}^{\trop}}$ does not contribute to the path-ordered product $\Theta_{\cu{W}_{\bb{L}^{\trop}},\partial\Delta}$. But actually, they do in the way that they help us to determine the soliton set for boundary walls.
\end{remark}

%By finiteness of walls on the boundary, there exists a small neighbourhood $V_{\check{\sigma}}\subset\partial\Delta$ of each vertex $\check{\sigma}$ so that $V_{\check{\sigma}}$ does not contain any endpoints of boundary walls. For a pair of maximal cones $\sigma_1,\sigma_2\in\Sigma$ that correspond to the vertices $\check{\sigma}_1,\check{\sigma}_2\in\Delta\cap M$, let $\gamma:[0,1]\to\partial\Delta$ be a smooth embedded path going from $V_{\check{\sigma}_1}-\check{\sigma}_1$ to $V_{\check{\sigma}_2}-\check{\sigma}_2$. By Theorem \ref{thm:consistent}, the path order product $G_{\sigma_1\sigma_2}$ associated to $\gamma$ 
%is in fact independent of $\gamma$ hence well-defined. By Theorem \ref{thm:consistent} again, $(G_{\sigma_1\sigma_2})_{\sigma_1,\sigma_2\in\Sigma(2)}$ is a 1-cocycle and hence defines

We now construct a toric vector bundle on $X_{\Sigma}$ by specifying its transition maps. Let $\sigma_1,\sigma_2$ be maximal cones and $\gamma_{\sigma_1\sigma_2}:[0,1]\to \partial\Delta_{\lambda}$ be the edge from the vertex $\lambda\check{\sigma}_1\subset\Delta_{\lambda}$ to the vertex $\lambda\check{\sigma}_2\subset\Delta_{\lambda}$. Then we define $G_{\sigma_1\sigma_2}$ to be the path order product of $\gamma_{\sigma_1\sigma_2}$ induced by parallel transports, $G_{\sigma_1\sigma_2}^{\mathrm{sf}}$ and $\Theta_w$'s with $w\in\cu{W}_{\bb{L}}$. However, the boundary points $\lambda\check{\sigma}_1,\lambda\check{\sigma}_2$ of $\gamma$ may contain in a boundary wall. When this happens, we use the orientation of $\gamma_{\sigma_1\sigma_2}$ to determine how they should be included/excluded:
\begin{enumerate}
    \item [(a)] If $\gamma_{\sigma_1\sigma_2}$ is orientation preserving, then we include the wall-crossing factor with the corresponding wall intersecting $\lambda\check{\sigma}_1$ and exclude the one that intersect $\lambda\check{\sigma}_2$.
    \item [(b)] If $\gamma_{\sigma_1\sigma_2}$ is orientation reversing, then we include the wall-crossing factor with the corresponding wall intersecting $\lambda\check{\sigma}_2$ and exclude the one that intersect $\lambda\check{\sigma}_1$. 
\end{enumerate}
In particular, we have $G_{\sigma_1\sigma_2}=G_{\sigma_2\sigma_1}^{-1}$. Proposition \ref{prop:consistent} then implies $(G_{\sigma_1\sigma_2})_{\sigma_1,\sigma_2\in\Sigma(2)}$ forms a 1-cocycle and hence gives a toric vector bundle $\cu{E}_{\lambda}(\cu{W}_{\bb{L}^{\trop}},\cu{L}^{\tw})$ whose tropicalization is $\bb{L}^{\trop}$. Via the identification $\cu{M}^{\tw}(\til{L},GL_1)\cong\cu{M}(L,GL_1)$, we write $\cu{E}_{\lambda}(\cu{W}_{\bb{L}^{\trop}},\cu{L})$ for the toric vector bundle $\cu{E}_{\lambda}(\cu{W}_{\bb{L}^{\trop}},\cu{L}^{\tw})$.

\begin{proposition}
    The toric vector bundle $\cu{E}_{\lambda}(\cu{W}_{\bb{L}^{\trop}},\cu{L})$ is independent of $\lambda\sim 1$ and the chosen points $\{x_{\rho}\}_{\rho\in\Sigma(1)}$.
\end{proposition}
\begin{proof}
    It is clear that the path-ordered product $G_{\sigma_1\sigma_2}$ is independent of $\lambda$ as long as $\lambda$ is close to 1. Let  $\{x_{\rho}\},\{x_{\rho}'\}$ be two set of generic points. If $w(1)$ is an endpoint that lies between the segment $[x_{\rho},x_{\rho}']\subset\check{\rho}$ joining $x_{\rho}$ and $x_{\rho}'$, the wall-crossing factors are related by
    $$\Theta_wG_{\sigma_1\sigma_2}^{\mathrm{sf}}=G_{\sigma_1\sigma_2}^{\mathrm{sf}}\Theta_w'.$$
    Hence the path-ordered product along $\gamma_{\sigma_1\sigma_2}$ stays the same.
\end{proof}

\begin{remark}
    It can be further shown that $\cu{E}_{\lambda}(\cu{W}_{\bb{L}^{\trop}},\cu{L})$ is independent of the convention (a), (b) above in the sense that we can include/exclude some of them in the path order product as long as they just show up once. This is because those wall-crossing automorphisms that correspond to a boundary wall with $w(1)=\check{\sigma}_w$ are automorphisms on the affine chart $U(\sigma_w)$.
\end{remark}

We now get a well-defined toric vector bundle $\cu{E}(\cu{W}_{\bb{L}^{\trop}},\cu{L})$ on $X_{\Sigma}$ defined over $\Bbbk$. By varying $\cu{L}$, we get a $b_1(L)$-dimensional family of such toric vector bundles.

\begin{theorem}\label{thm:TVB_via_SN}
    Let $\bb{L}^{\trop}$ be a tropical Lagrangian multi-section over $\Sigma$. Suppose there exists a spectral network $\cu{W}_{\bb{L}^{\trop}}$ on $\Delta$ subordinate to a simply branched covering map $p:L\to\Delta$ with flag data given by \eqref{eqn:slope_condition}. Then for any $\Bbbk^{\times}$-local system on $L$, there exists a toric vector bundle $\cu{E}(\cu{W}_{\bb{L}^{\trop}},\cu{L})$ on $X_{\Sigma}$ with tropicalization $\bb{L}^{\trop}$. Furthermore, the assignment $\Psi_{\cu{W}_{\bb{L}^{\trop}}}:\cu{L}\mapsto\cu{E}(\cu{W}_{\bb{L}^{\trop}},\cu{L})$ gives an injection from $\cu{M}(L,GL_1)$, the moduli space of rank 1 local systems on $L$, to $\cu{M}(X_{\Sigma},\bb{L}^{\trop})$, the moduli space of toric vector bundles with tropicalization $\bb{L}^{\trop}$.
\end{theorem}
\begin{proof}
    It remains to prove that $\cu{M}(L,GL_1)\cong\cu{M}^{\tw}(\til{L},GL_1)$ injects to the moduli space of toric vector bundles with tropicalization $\bb{L}^{\trop}$. Suppose $\cu{E}(\cu{W}_{\bb{L}^{\trop}},\cu{L})\cong\cu{E}(\cu{W}_{\bb{L}^{\trop}},\cu{L}')$ as toric vector bundles. Then they are isomorphic as $T$-equivariant bundles on the toric boundary divisor $\partial X_{\Sigma}$ after restriction. By separability of $\bb{L}^{\trop}$, these ``boundary" bundles have $G_{\rho}^{\mathrm{sf}}$'s as transition maps. We then obtain the relation
    $$P_{\til{x}_1^{(\alpha)}\to\til{x}_2^{(\alpha')}}h_{\sigma_2^{(\alpha')}\sigma_2^{(\alpha')}}=h_{\sigma_1^{(\alpha)}\sigma_1^{(\alpha)}}P_{\til{x}_1^{(\alpha)}\to \til{x}_2^{(\alpha')}}',$$
    for some non-zero constant $h_{\sigma_1^{(\alpha)}\sigma_1^{(\alpha)}},h_{\sigma_2^{(\alpha')}\sigma_2^{(\alpha')}}$ whenever $\sigma_1^{(\alpha)}\cap\sigma_2^{(\alpha')}\in\Sigma_{L^{\trop}}(1)$. This means the parallel transports of $\cu{L}^{\tw},{\cu{L}^{\tw}}'$ along any curves in $L$ are the same after rescaling their frames by non-zero constants. This shows that $\cu{L}^{\tw}\cong{\cu{L}^{\tw}}'$ as twisted-flat local systems.
\end{proof}

By Lemma \ref{lem:boundary_wall}, if $\cu{W}_{\bb{L}}$ is a spectral network, it must respect the flag data given in \eqref{eqn:slope_condition}. Hence Theorem \ref{thm:TVB_via_SN} implies the following

\begin{corollary}\label{cor:TVB_via_SN}
    Let $\bb{L}$ be an $r$-fold $\Lambda_{\bb{L}^{\trop}}$-admissible Lagrangian multi-section and suppose $\cu{W}_{\bb{L}}$ is a non-degenerated spectral network subordinate. Then for any $\Bbbk^{\times}$-local system $\cu{L}$ on the domain $L$, there exists a rank $r$ toric vector bundle $\cu{E}(\bb{L},\cu{L})$ whose tropicalization is $\bb{L}^{\trop}$ and the assignment $\Psi_{\cu{W}_{\bb{L}}}:\cu{L}\mapsto\cu{E}(\bb{L},\cu{L})$ is an injection.   
\end{corollary}

Combining Corollary \ref{cor:TVB_via_SN} with Theorem \ref{thm:existence}, we obtain

\begin{corollary}\label{cor:well_behave_TVB}
    If $\bb{L}$ is a well-behaved $r$-fold $\Lambda_{\bb{L}^{\trop}}$-admissible Lagrangian multi-section, then for any $\Bbbk^{\times}$-local system $\cu{L}$ on the domain $L$, there exists a rank $r$ toric vector bundle $\cu{E}(\bb{L},\cu{L})$ whose tropicalization is $\bb{L}^{\trop}$ and the assignment $\Psi_{\cu{W}_{\bb{L}}}:\cu{L}\mapsto\cu{E}(\bb{L},\cu{L})$ is an injection.  
\end{corollary}

%\begin{remark}
   % We don't expect the assignment $\cu{L}\mapsto\cu{E}(\bb{L},\cu{L})$ to be surjective in general as the moduli of toric vector bundles can be ``arbitrarily" singular \cite{Payne_moduli_of_TVB}. However, we expect surjectivity if one runs over all possible unobstructed realizations of $\bb{L}^{\trop}$.
%\end{remark}

\begin{remark}
    As mentioned in Remark \ref{rmk:grad_trees}, we think of the gradient trees are the tropical limit of Maslov index zero holomorphic disks. Moreover, Nho has proved in \cite{Nho_spectral_network_family_fleor} that non-abelianization is same as the family Floer assignment $\xi\mapsto HF((\bb{L},\cu{L}),T^*_{\xi}N_{\bb{R}})$, it is  very natural to expect $\cu{E}(\bb{L},\cu{L})$ is the equivariant mirror bundle of $(\bb{L},\cu{L})$.
\end{remark}

\section{Examples and applications}\label{sec:existence}

In this section, we will provide some examples of spectral networks subordinate to $\Lambda_{\bb{L}^{\trop}}$-admissible Lagrangian multi-sections, and will end by giving two applications of our construction.

\subsection{Examples}

As mentioned in Section \ref{sec:spec_net_L}, the work of \cite{OS} provides us many examples of non-degenerated networks so that our construction in Section \ref{sec:mirror_construction} can be applied. We now give some examples of non-degenerated spectral networks subordinate to those 2-fold Lagrangian multi-sections constructed in \cite[Theorem 5.8]{OS}, one non-example, and one 3-fold example.

To begin, we recall that in \cite[Section 5.1]{OS}, Oh and the author associated each 2-fold tropical Lagrangian multi-section with an important integer, which we now recall. Given a 2-fold tropical Lagrangian multi-section $\bb{L}^{\trop}$, let $\gamma\in\text{Deck}_{L^{\trop}/N_{\bb{R}}}(p)$ be the unique deck transformation over $N_{\bb{R}}$. Let $C:=p^{-1}(S^1)$ be the preimage of the unit circle $S^1\subset N_{\bb{R}}$ centred at the origin. If $L^{\trop}\cong\bb{R}^2$, parametrize $C$ by $[0,2\pi)$ while if $L^{\trop}\cong\bb{R}^2\sqcup\bb{R}^2$, $C=C^+\sqcup C^-$ and we also parametrize $C^{\pm}$ by $[0,2\pi)$. The number $N_{\bb{L}^{\trop}}$ is defined to be
$$N_{\bb{L}^{\trop}}:=\begin{dcases}
    \#\text{Graph}(\varphi^{\trop}|_{[0,\pi)})\cap\text{Graph}(\varphi^{\trop}\circ\gamma|_{[0,\pi)}) & \text{ if }L^{\trop}=\bb{R}^2,\\
    \#\text{Graph}(\varphi^{\trop}|_{[0,2\pi)})\cap\text{Graph}(\varphi^{\trop}\circ\gamma|_{[0,2\pi)}) & \text{ if }L^{\trop}=\bb{R}^2\sqcup\bb{R}^2.
\end{dcases}$$
Separability ensures the intersections are transversal and only occur at interior of maximal cones. One of the main results in \cite{OS} is that a 2-fold tropical Lagrangian multi-section can be realised by a connected embedded Lagrangian multi-section if $N_{\bb{L}^{\trop}}\geq 3$. As shown in \cite{OS}, the number of branch points of such realization is $N-2$, so the topology of an embedded realization is completely determined, namely, any embedded realization must have $b_1=N_{\bb{L}^{\trop}}-3$.
        
\begin{example}\label{eg:rank2_N_generic}
    Let $\sigma_1,\dots,\sigma_{N_{\bb{L}^{\trop}}}\in\Sigma(2)$ be the maximal cones where the two graphs have an intersection in $\Int(\sigma_i)\cap S^1$, for $i=1,\dots,N_{\bb{L}^{\trop}}$. They are dual to a collection of vertices $\check{\sigma}_1,\dots,\check{\sigma}_{N_{\bb{L}^{\trop}}}$ in the polytope $\Delta$. We colour them by black in Figure \ref{fig:examples}. The slope condition \eqref{eqn:slope_condition} gives each edge of $\Delta$ a label, 12 or 21. The label changes from $\alpha\beta$ to $\beta\alpha$ whenever an $\alpha\beta$-edge are connected by a black vertices or passing through a branch cut of $\check{p}_L:L\to\ring{\Delta}$. The following are two examples of non-degenerated spectral networks subordinate to these 2-fold Lagrangian multi-sections.
        \begin{figure}[H]
	       \centering
	       \includegraphics[width=120mm]{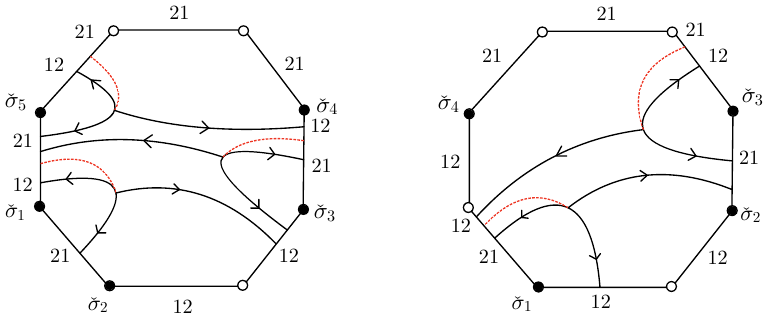}
	       \caption{Two examples of spectral networks subordinate to 2-fold Lagrangian multi-sections constructed in \cite{OS}. The black vertices correspond to those maximal cones for which the two graphs intersect. The Lagrangian multi-section on the left-handed figure realizes a $5$-generic tropical Lagrangian multi-section and the one on the right-handed side realizes a $4$-generic one.}
	       \label{fig:examples}
        \end{figure}
        When $X_{\Sigma}=\bb{P}^2$, we have shown in \cite[Section 6]{OS} that the mirror of any indecomposable rank 2 toric vector bundle on $\bb{P}^2$ is an embedded 2-fold Lagrangian multi-section that is diffeomorphic to $\bb{R}^2$ ($N=3$). In particular, there's only one branch point. The spectral networks subordinate to these Lagrangian multi-sections consist of exactly 3 walls emitting from the branch point and each edge has exactly one wall hitting its relative interior. See Figure \ref{fig:P2_example}
        \begin{figure}[H]
	       \centering
	       \includegraphics[width=40mm]{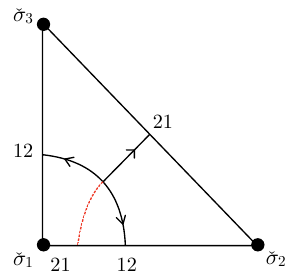}
	       \caption{The spectral network subordinate to the mirror Lagrangian multi-section of a rank 2 indecomposable toric vector bundle on $\bb{P}^2$.}
	       \label{fig:P2_example}
        \end{figure}
\end{example}

    \begin{example}\label{eg:non_example}
        We provide an interesting non-example. Consider the 2-fold tropical Lagrangian multi-section over $\Sigma_{\bb{P}^2}$ given in \cite[Example 5.22]{Suen_trop_lag}, which cannot be B-realizable. This tropical Lagrangian multi-section is 1-generic. One can realize $\bb{L}^{\trop}$ by an immersed Lagrangian multi-section $\bb{L}$ with a simply connected domain. But $\bb{L}$ is actually obstructed by a non-index 1 immersed point. The LT lines of this Lagrangian are depicted in Figure \ref{fig:non_example}.
        \begin{figure}[H]
	       \centering
	       \includegraphics[width=40mm]{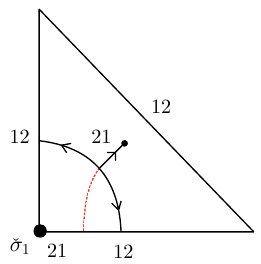}
	       \caption{A 21-LT line terminates at a non-index 1 critical point.}
	       \label{fig:non_example}
        \end{figure}
        This configuration does not give a spectral network around the non-index 1 critical point. Floer theoretically, the flow line from the branch point to that critical should be viewed as a fishtail-like holomorphic disk bounded by $\bb{L}$ passing through the immersed point. This holomorphic disk obstructs the Floer theory of $\bb{L}$.
    \end{example}

\begin{example}
  We also give a rank 3 example. The tropical Lagrangian multi-section $\bb{L}^{\trop}$ over $\Sigma_{\bb{P}^2}$ as shown in Figure \ref{fig:trop_lag_rank3}
      \begin{figure}[H]
	       \centering
	       \includegraphics[width=120mm]{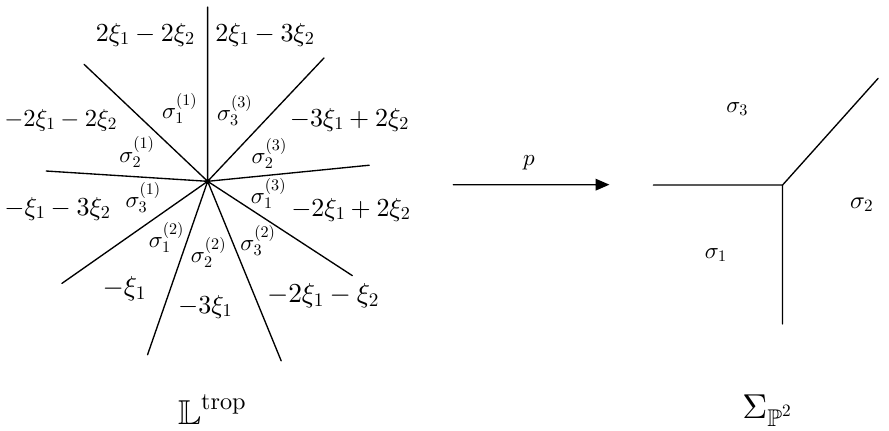}
	       \caption{The 3-fold tropical Lagrangian multi-section $\bb{L}^{\trop}$ over $\Sigma_{\bb{P}^2}$. }
	       \label{fig:trop_lag_rank3}
    \end{figure}
%By calculating the slope differences $m(\sigma_i^{(\alpha)})-m(\sigma_i^{(\beta)})$, for $i=1,2,3$ and $\alpha,\beta=1,2,3$, we find that
   % \begin{align*}
      %  m(\sigma_1^{(1)})-m(\sigma_1^{(2)}),\,m(\sigma_1^{(2)})-m(\sigma_1^{(3)}),\,m(\sigma_1^{(1)})-m(\sigma_1^{(3)})\in(\sigma_1\cap\sigma_2)^{\vee}\cap M,\\
      %  m(\sigma_2^{(2)})-m(\sigma_2^{(1)}),\,m(\sigma_2^{(3)})-m(\sigma_2^{(2)}),\,m(\sigma_2^{(3)})-m(\sigma_2^{(1)})\in(\sigma_2\cap\sigma_3)^{\vee}\cap M,\\
      %  m(\sigma_3^{(1)})-m(\sigma_3^{(2)}),\,m(\sigma_3^{(2)})-m(\sigma_3^{(3)}),\,m(\sigma_3^{(1)})-m(\sigma_3^{(3)})\in(\sigma_3\cap\sigma_1)^{\vee}\cap M
 %   \end{align*}
 It was shown in \cite[Example 5.22]{OS} that this tropical Lagrangian multi-section can be realized by a genus 1 embedded 3-fold Lagrangian multi-section $\bb{L}$ with one cylindrical end, branched over $N_{\bb{R}}$ at two non-simple branch points. By perturbing the two non-simple branched points, we obtain a spectral network as shown on the right-handed side of Figure \ref{fig:rank3_SN}. Note that in this 3-fold example, scattering occurs.
    \begin{figure}[H]
	       \centering
	       \includegraphics[width=120mm]{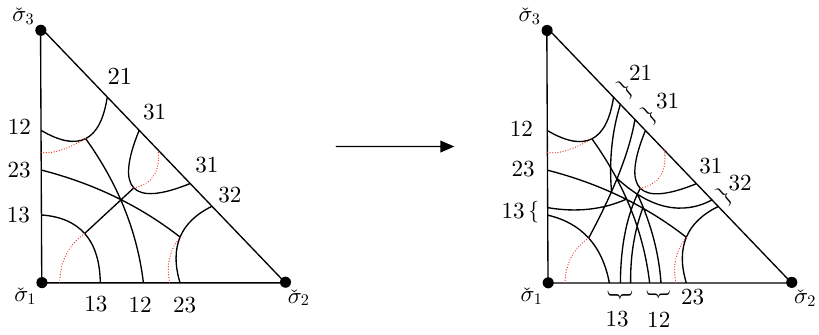}
	       \caption{The left-handed side shows a degenerated spectral network subordinate to a 3-fold Lagrangian multi-section while the right-handed side is a non-degenerated one, obtained by perturbing the left-handed one.}
	       \label{fig:rank3_SN}
    \end{figure}
\end{example}

\subsection{Applications}

We now give two applications on Theorem \ref{thm:TVB_via_SN}; one on the $A$-side and one on the $B$-side.

The first application is based on Example \ref{eg:rank2_N_generic}, which is an enhancement of \cite[Corollary 5.16]{OS}.

    \begin{theorem}\label{thm:A_real}
        A 2-fold tropical Lagrangian multi-section can be realized by a connected embedded exact Lagrangian multi-section in $T^*M_{\bb{R}}$ if and only if $N_{\bb{L}^{\trop}}\geq 3$.
    \end{theorem}
    \begin{proof}
        When $N_{\bb{L}^{\trop}}$ is odd, this is essentially \cite[Corollary 5.16]{OS}. When $N_{\bb{L}^{\trop}}$ is even, in the same paper, we have already shown that $\bb{L}^{\trop}$ can be realized by a connected embedded Lagrangian multi-section if $N_{\bb{L}^{\trop}}\geq 4$. It remains to prove exactness and the converse. For exactness, we can first choose the local model $L_{f_d}$ in \cite[Section 5.2.2]{OS} to be exact with respect to the canonical Liouville form $\lambda_{N_{\bb{R}}}$ on $T^*N_{\bb{R}}$ (See \cite[Remark 3.10]{Nho_spectral_network_family_fleor} for the existence of such exact Lagrangians). Hence we can make the Lagrangian multi-sections constructed in \cite[Theorem 5.15]{OS} to be $\lambda_{N_{\bb{R}}}$-exact. Since $\lambda_{M_{\bb{R}}}+\lambda_{N_{\bb{R}}}=df$ with $f(x,\xi):=\inner{x,\xi}$, these multi-sections are also $\lambda_{M_{\bb{R}}}$-exact (See Remark \ref{rmk:finite} for the primitive of $\bb{L}$). If $\bb{L}^{\trop}$ can be realized by a connected embedded Lagrangian multi-section $\bb{L}$, then there is a spectral network subordinate to $\bb{L}$. If $N=2$, there exist two black vertices as in Figure \ref{fig:examples} and dividing the set of edges into two parts as shown in the figure on the left in Figure \ref{fig:N_leq_2}.
        \begin{figure}[H]
	       \centering
	       \includegraphics[width=120mm]{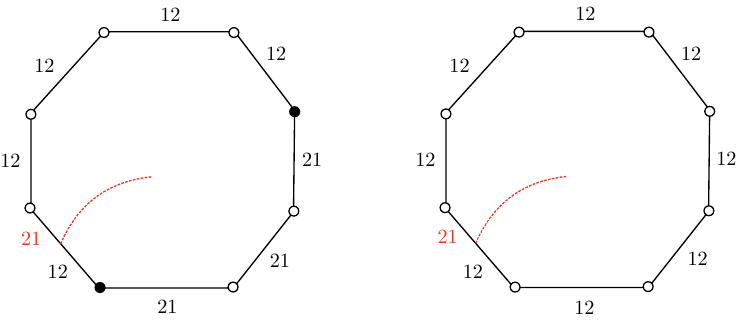}
	       \caption{The case $N\leq 2$. The red dotted line represents a branch cut, causing a label change and thus no consistent labelling.}
	       \label{fig:N_leq_2}
        \end{figure}
        However, as there exists at least one branch point, there is no consistent labelling. The case when $N=0$ is similar. This argument also applies when $N_{\bb{L}^{\trop}}$ is odd, hence showing that $N_{\bb{L}^{\trop}}\geq 3$ is a necessary condition for the $A$-realization problem \emph{without} using mirror symmetry.
    \end{proof}

    We now move to our second application. It was suggested in \cite{Spectral_networks} that (at when $r=2$) cluster structure \cite{FZ_Cluster1, FZ_Clustert2, BFZ_Cluster3} arise by connecting two non-degenerated spectral networks through a degenerated one, so we expect by passing from $\cu{W}_{\bb{L}}$ to $\cu{W}_{\bb{L}'}$ through a degenerated one, we should be able to obtain cluster structure on $\cu{M}(X_{\Sigma},\bb{L}^{\trop})$.
    
    We first recall the notion of $\cu{X}$-cluster structure (without frozen variables). Let $Q$ be an acyclic (without 1- or 2-cycles) quiver with a finite vertex set $V(Q)$ and $B$ be the \emph{exchange matrix} given by
    $$B_{ij}:=\#(i\to j)-\#(j\to i),$$
    with $i,j\in V(Q)$. The data $(Q,B)$ is called a \emph{seed}. Note that $B$ and $Q$ determine each other. For each $k\in V(Q)$, the $k$-mutation $(\mu_k(Q),\mu_k(B))$ is defined via the exchange matrix
        $$\mu_k(B)_{ij}:=\begin{cases}
            -B_{ij} & \text{ if }k=i,j,\\
            B_{ij}+\frac{1}{2}(B_{ik}|B_{kj}|+|B_{ik}|B_{kj}) & \text{ otherwise}.
        \end{cases}$$
    For each vertex $i\in V(Q)$, we assign a formal invertible variable $X_i$, and for a vertex $k\in V(Q)$, we define the birational map
    $$\mu_k^*X_i:=
    \begin{cases}
        X_k^{-1} & \text{ if }i=k\\
        X_i(1+X_k^{-\text{sgn}(B_{ik})})^{-B_{ik}} & \text{ if }i\neq k.
    \end{cases}$$
    The \emph{$\cu{X}$-cluster variety associate to $Q$} is defined to be the gluing of $\Spec(\Bbbk[X_i^{\pm 1}|i\in \mu_k(Q)])\cong(\Bbbk^{\times})^{\#V(Q)}$ with respect to the mutations:
    $$\cu{X}(Q):=\lim_{\substack{\longrightarrow\\k}}\Spec(\Bbbk[X_i^{\pm 1}\,|\,i\in \mu_k(Q)]).$$

    \begin{theorem}\label{thm:cluster}
    Let $\bb{L}^{\trop}$ be a 2-dimensional 2-fold tropical Lagrangian multi-section with $N_{\bb{L}^{\trop}}\geq 3$. Then the moduli space $\cu{M}(X_{\Sigma},\bb{L}^{\trop})$ is an $A_{N_{\bb{L}^{\trop}}-3}$-type $\cu{X}$-cluster variety.
    \end{theorem}
    \begin{proof}
        Throughout the proof, we use the identification $\cu{M}(L,GL_1)\cong\cu{M}^{\tw}(\til{L},GL_1)$ and write $\cu{L}^{\tw}\in\cu{M}(\til{L},GL_1)$ for the image of $\cu{L}\in\cu{M}(L,GL_1)$ under this identification.
        
        It is known in \cite[Section 5]{HN_SN_FN_coordinates} that the non-abelianization map $$\ol{\Psi}_{\cu{W}}:\cu{M}(L,GL_1)\to\cu{M}_F(\Delta,p_1,\dots,p_{N_{\bb{L}^{\trop}}},GL_2)$$
        has open dense image, where $\cu{M}_F(\Delta,p_1,\dots,p_{N_{\bb{L}^{\trop}}},GL_2)$ is the moduli space of \emph{framed rank 2 local systems} on a 2-disk with boundary markings $p_1,\dots,p_{N_{\bb{L}^{\trop}}}$ in the sense of Fock-Goncharov \cite{FG_moduli_space}. The framing is given by taking flat subspaces of the non-abelianization at the markings. See \cite[Section 10.2]{Spectral_networks} for details. There is an open embedding
        $$\cu{M}(X_{\Sigma},\bb{L}^{\trop})\hookrightarrow\cu{M}_F(\Delta,p_1,\dots,p_{N_{\bb{L}^{\trop}}},GL_2).$$
        This embedding is obtained by first forgetting the monomial $z^{m(\sigma_1^{(\alpha)})-m(\sigma_2^{(\beta)})}$ in the transition maps of $\cu{E}(\bb{L},\cu{L})$ and mapping the resulting rank 2 local system together with the invariant flag (which is a line when $r=2)$ to $\cu{M}_F(\Delta,p_1,\dots,p_{N_{\bb{L}^{\trop}}},GL_2)$. This gives, for any $\Lambda_{\bb{L}^{\trop}}$-admissible embedded Lagrangian multi-section $\bb{L}$, a triangle of embeddings
        \[
        \begin{tikzcd}
            & \cu{M}(L,GL_1) \arrow[dl,"\Psi_{\cu{W}_{\bb{L}}}"'] \arrow{dr}{\ol{\Psi}_{\cu{W}_{\bb{L}}}} \\
            \cu{M}(X_{\Sigma},\bb{L}^{\trop}) \arrow[rr,hook] && \cu{M}_F(\Delta,p_1,\dots,p_{N_{\bb{L}^{\trop}}},GL_2)
        \end{tikzcd}
        \]
        In particular, the image of $\Psi_{\cu{W}_{\bb{L}}}:\cu{M}(L,GL_1)\to\cu{M}(X_{\Sigma},\bb{L}^{\trop})$
        is also a dense open subset.
        
        Let $V(Q)=\{\gamma_i\}_{i=1}^{N_{\bb{L}^{\trop}}-3}$ be the set of homotopy classes of paths in $\Delta$ joining two branched points that do not pass through $|\cu{W}_{\bb{L}}|$. With a choice of orientation, each $\gamma_i$ lifted uniquely to a cycle $\til{\gamma}_i\subset\til{L}$ and their projection to $L$ form a basis of $H_1(L;\bb{Z})$. The exchange matrix $B$ is given by $B_{ij}:=\inner{\til{\gamma}_i,\til{\gamma}_j}=\pm 1$ and it determines a quiver $Q$ with vertex set $V(Q)$. The pair $(Q,B)$ defines a seed. Define
        $$X_i:=X_{\gamma_i}:=\hol_{\cu{L}^{\tw}}(\til{\gamma}_i)\in\Bbbk^{\times}.$$
        These coordinates give an identification $\cu{M}(L,GL_1)\cong(\Bbbk^{\times})^{N_{\bb{L}^{\trop}}-3}$ and hence coordinates
        $$(\Bbbk^{\times})^{N_{\bb{L}^{\trop}}-3}\cong\cu{M}(L,GL_1)\to\cu{M}(X_{\Sigma},\bb{L}^{\trop})$$
        Suppose we have two embedded $A$-realizations $\bb{L},\bb{L}'$ (they share a common domain $L$) of $\bb{L}^{\trop}$ so that locally around three branch points, their spectral networks are related as shown in Figure \ref{fig:mutation_tex}. Suppose $\cu{L},\cu{L}'\in\cu{M}(L,GL_1)$ are $\Bbbk^{\times}$-local systems on $L$ such that $\cu{E}(\bb{L},\cu{L})=\cu{E}(\bb{L}',\cu{L}')$. Let $X_i':=\hol_{\cu{L}'}(\til{\gamma}_i)$.
        \begin{figure}[H]
	       \centering
	       \includegraphics[width=120mm]{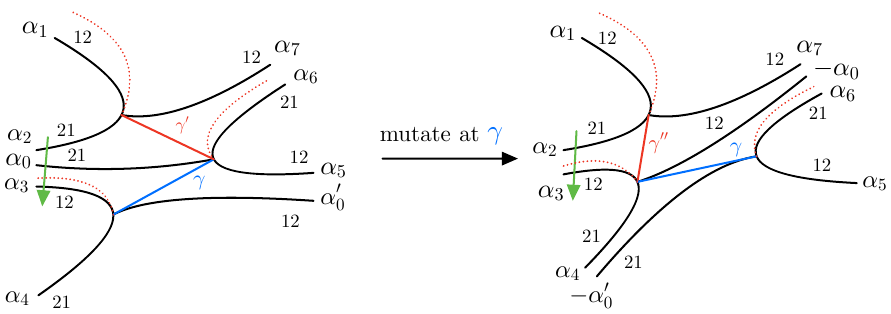}
	       \caption{For the lifts $\til{\alpha}_i\subset\til{L}$ of the paths $\alpha_i\subset\Delta$'s, we have $\til{\alpha}_1=-\til{\alpha}_2$, $\til{\alpha}_3=-\til{\alpha}_4$, and $\til{\alpha}_5=-\til{\alpha}_6$. Mutation is performed at $\gamma$. The wall-crossing factors across the green arrows are the same, giving us the cluster transformations.}
	       \label{fig:mutation_tex}
        \end{figure}
        Note that the pairs of paths $\alpha_1,\alpha_2$ and $\alpha_3,\alpha_4$ must be separated by black vertices, say by $\check{\sigma}_1,\check{\sigma}_2$, respectively. In particular, the composition of the wall-crossing automorphisms along the green arrows gives the transition map from the affine chart $U(\sigma_1)$ to $U(\sigma_2)$. By assumption, $\cu{E}(\bb{L},\cu{L})=\cu{E}(\bb{L}',\cu{L}')$, the wall-crossing factors along the two green paths are the same. By our construction in Section \ref{sec:mirror_construction}, we have
        $$\begin{pmatrix}
            1 & \ol{P}_{\cu{L}^{\tw}}(\til{\alpha}_3)\\
            0 & 1
        \end{pmatrix}
        \begin{pmatrix}
            0 & 1\\
            1 & 0
        \end{pmatrix}
        \begin{pmatrix}
            1 & 0\\
            \ol{P}_{\cu{L}^{\tw}}(\til{\alpha}_0) & 1
        \end{pmatrix}
        \begin{pmatrix}
            1 & 0\\
            \ol{P}_{\cu{L}^{\tw}}(\til{\alpha}_2) & 1
        \end{pmatrix}
        =\begin{pmatrix}
            1 & \ol{P}_{\cu{L}^{\tw'}}(\til{\alpha}_3)\\
            0 & 1
        \end{pmatrix}\begin{pmatrix}
            0 & 1\\
            1 & 0
        \end{pmatrix}
        \begin{pmatrix}
            1 & 0\\
            \ol{P}_{\cu{L}^{\tw'}}(\til{\alpha}_2) & 1
        \end{pmatrix}.$$
        The off-diagonal matrices indicate the change in branches when the green paths pass through the branch cut (red-dotted line). Hence we obtain the relation
        $$\ol{P}_{\cu{L}^{\tw'}}(\til{\alpha}_2)+\ol{P}_{\cu{L}^{\tw'}}(\til{\alpha}_3)=\ol{P}_{\cu{L}^{\tw}}(\til{\alpha}_2)+\ol{P}_{\cu{L}^{\tw}}(\til{\alpha}_0)+\ol{P}_{\cu{L}^{\tw}}(\til{\alpha}_3).$$
        But $\til{\alpha}_1=-\til{\alpha}_2$ and $\ol{P}_{\cu{L}^{\tw}}(\til{\alpha}_1)=\ol{P}_{\cu{L}^{\tw'}}(\til{\alpha}_1)$ imply
        $$\ol{P}_{\cu{L}^{\tw'}}(\til{\alpha}_3)=\ol{P}_{\cu{L}^{\tw}}(\til{\alpha}_0)+\ol{P}_{\cu{L}^{\tw}}(\til{\alpha}_3).$$
        Using the relations $\til{\gamma}=\til{\alpha}_3-\til{\alpha}_0$, $\til{\gamma}'=\til{\alpha}_0-\til{\alpha}_2$ and $\til{\gamma}''=\til{\alpha}_3-\til{\alpha}_2$, we obtain
        $$\hol_{\cu{L}^{\tw'}}(\til{\gamma}'')=\hol_{\cu{L}^{\tw}}(\til{\gamma}')(1+\hol_{\cu{L}^{\tw}}(\til{\gamma})),$$
        which is the mutation of the variable $X_{\gamma'}$ at $\gamma$. One can also check that
        $$\hol_{\cu{L}^{\tw'}}(\til{\gamma})=\hol_{\cu{L}^{\tw}}(\til{\gamma})^{-1},$$
        which is the mutation of $X_{\gamma}$ at $\gamma$. On the other hand, one can also see that if $\gamma'$ is invariant under mutation at $\gamma$ as shown in Figure \ref{fig:mutation2}, we also have
        $$\hol_{\cu{L}^{\tw'}}(\til{\gamma}')=\hol_{\cu{L}^{\tw}}(\til{\gamma}')(1+\hol_{\cu{L}^{\tw}}(\til{\gamma})).$$
        \begin{figure}[H]
	       \centering
	       \includegraphics[width=120mm]{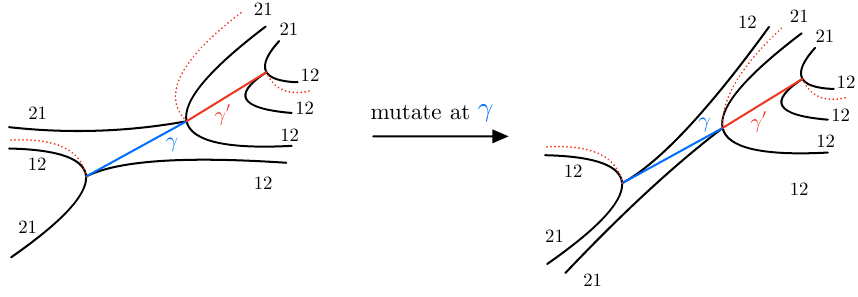}
	       \caption{$\gamma'$ is invariant under mutation at $\gamma$.}
	       \label{fig:mutation2}
        \end{figure}
        These calculations prove that $\{X_i\}_{i=1}^{N_{\bb{L}^{\trop}}-3}$ and $\{X_i'\}_{i=1}^{N_{\bb{L}^{\trop}}-3}$ are related by the $\cu{X}$-cluster mutation.
    \end{proof}

    \begin{remark}
        As suggested by \cite{Spectral_networks}, non-degenerated spectral networks should provide cluster transformations when they are connected through degenerated spectral networks. We, therefore, believe cluster structure continues to exist on $\cu{M}(X_{\Sigma},\bb{L}^{\trop})$ when $r>2$. But the combinatorics and a detailed proof seem to be more involved and we should leave it for later investigation. 
    \end{remark}

	\bibliographystyle{amsalpha}
	\bibliography{geometry-oh}

\providecommand{\bysame}{\leavevmode\hbox to3em{\hrulefill}\thinspace}
\providecommand{\MR}{\relax\ifhmode\unskip\space\fi MR }
% \MRhref is called by the amsart/book/proc definition of \MR.
\providecommand{\MRhref}[2]{%
  \href{http://www.ams.org/mathscinet-getitem?mr=#1}{#2}
}
\providecommand{\href}[2]{#2}
\begin{thebibliography}{FOOO09}

\bibitem[Abo09]{Abouzaid09}
M.~Abouzaid, \emph{Morse homology, tropical geometry, and homological mirror symmetry for toric varieties}, Selecta Math. (N.S.) \textbf{15} (2009), no.~2, 189--270.

\bibitem[AJ10]{AJ}
M.~Akaho and D.~Joyce, \emph{Immersed {L}agrangian {F}loer theory}, J. Differential Geom. \textbf{86} (2010), no.~3, 381--500.

\bibitem[BFZ05]{BFZ_Cluster3}
Arkady Berenstein, Sergey Fomin, and Andrei Zelevinsky, \emph{Cluster algebras. {III}. {U}pper bounds and double {B}ruhat cells}, Duke Math. J. \textbf{126} (2005), no.~1, 1--52. \MR{2110627}

\bibitem[CMS22]{CMS_k3bundle}
Kwokwai Chan, Ziming~Nikolas Ma, and Yat-Hin Suen, \emph{Tropical {L}agrangian multi-sections and smoothing of locally free sheaves over degenerate {C}alabi-{Y}au surfaces}, Adv. Math. \textbf{401} (2022), Paper No. 108280. \MR{4387852}

\bibitem[CS19]{CS_SYZ_imm_Lag}
K.~Chan and Y.-H. Suen, \emph{S{YZ} transforms for immersed {L}agrangian multisections}, Trans. Amer. Math. Soc. \textbf{372} (2019), no.~8, 5747--5780.

\bibitem[FG06]{FG_moduli_space}
Vladimir Fock and Alexander Goncharov, \emph{Moduli spaces of local systems and higher {T}eichm\"uller theory}, Publ. Math. Inst. Hautes \'Etudes Sci. (2006), no.~103, 1--211. \MR{2233852}

\bibitem[FLTZ12]{CCC_HMS}
B.~Fang, C.-C.~M. Liu, D.~Treumann, and E.~Zaslow, \emph{T-duality and homological mirror symmetry for toric varieties}, Adv. Math. \textbf{229} (2012), no.~3, 1875--1911.

\bibitem[FO97]{FO}
Kenji Fukaya and Yong-Geun Oh, \emph{Zero-loop open strings in the cotangent bundle and {M}orse homotopy}, Asian J. Math. \textbf{1} (1997), no.~1, 96--180.

\bibitem[FOOO09]{FOOO1}
Kenji Fukaya, Yong-Geun Oh, Hiroshi Ohta, and Kaoru Ono, \emph{Lagrangian intersection {F}loer theory: anomaly and obstruction. {P}art {I}}, AMS/IP Studies in Advanced Mathematics, vol.~46, American Mathematical Society, Providence, RI; International Press, Somerville, MA, 2009.

\bibitem[Fuk05]{Fukaya_asymptotic_analysis}
K.~Fukaya, \emph{Multivalued {M}orse theory, asymptotic analysis and mirror symmetry}, Graphs and patterns in mathematics and theoretical physics, Proc. Sympos. Pure Math., vol.~73, Amer. Math. Soc., Providence, RI, 2005, pp.~205--278.

\bibitem[FZ02]{FZ_Cluster1}
Sergey Fomin and Andrei Zelevinsky, \emph{Cluster algebras. {I}. {F}oundations}, J. Amer. Math. Soc. \textbf{15} (2002), no.~2, 497--529. \MR{1887642}

\bibitem[FZ03]{FZ_Clustert2}
\bysame, \emph{Cluster algebras. {II}. {F}inite type classification}, Invent. Math. \textbf{154} (2003), no.~1, 63--121. \MR{2004457}

\bibitem[GMN13]{Spectral_networks}
Davide Gaiotto, Gregory~W. Moore, and Andrew Neitzke, \emph{Spectral networks}, Ann. Henri Poincar\'{e} \textbf{14} (2013), no.~7, 1643--1731. \MR{3115984}

\bibitem[GS06]{GS1}
M.~Gross and B.~Siebert, \emph{Mirror symmetry via logarithmic degeneration data. {I}}, J. Differential Geom. \textbf{72} (2006), no.~2, 169--338.

\bibitem[GS10]{GS2}
\bysame, \emph{Mirror symmetry via logarithmic degeneration data, {II}}, J. Algebraic Geom. \textbf{19} (2010), no.~4, 679--780.

\bibitem[GS11]{GS11}
\bysame, \emph{From real affine geometry to complex geometry}, Ann. of Math. (2) \textbf{174} (2011), no.~3, 1301--1428.

\bibitem[Hic]{Hicks_realization}
Jeff Hicks, \emph{Realizability in tropical geometry and unobstructedness of {L}agrangian submanifolds}, Geometry and Topology, to appear, \href{https://arxiv.org/abs/2204.06432}{arXiv:2204.064325}.

\bibitem[Hic20]{Hicks_Trop_Lag_hyperseuface_unobs}
Jeffrey Hicks, \emph{Tropical {L}agrangian hypersurfaces are unobstructed}, J. Topol. \textbf{13} (2020), no.~4, 1409--1454. \MR{4125753}

\bibitem[Hic21]{Hicks_toric_del_Pezzo}
\bysame, \emph{Tropical {L}agrangians in toric del-{P}ezzo surfaces}, Selecta Math. (N.S.) \textbf{27} (2021), no.~1, Paper No. 3, 50. \MR{4198528}

\bibitem[HN16]{HN_SN_FN_coordinates}
Lotte Hollands and Andrew Neitzke, \emph{Spectral networks and {F}enchel-{N}ielsen coordinates}, Lett. Math. Phys. \textbf{106} (2016), no.~6, 811--877. \MR{3500424}

\bibitem[Kan75]{Kaneyama_classification}
Tamafumi Kaneyama, \emph{On equivariant vector bundles on an almost homogeneous variety}, Nagoya Math. J. \textbf{57} (1975), 65--86. \MR{376680}

\bibitem[LYZ00]{LYZ}
N.~C. Leung, S.-T. Yau, and E.~Zaslow, \emph{From special {L}agrangian to {H}ermitian-{Y}ang-{M}ills via {F}ourier-{M}ukai transform}, Adv. Theor. Math. Phys. \textbf{4} (2000), no.~6, 1319--1341.

\bibitem[Mat21]{Matessi_Lag_pants}
Diego Matessi, \emph{Lagrangian pairs of pants}, Int. Math. Res. Not. IMRN (2021), no.~15, 11306--11356. \MR{4294119}

\bibitem[Mik19]{Mikhalkin_trop_to_Lag}
Grigory Mikhalkin, \emph{Examples of tropical-to-{L}agrangian correspondence}, Eur. J. Math. \textbf{5} (2019), no.~3, 1033--1066. \MR{3993277}

\bibitem[MR20]{Mak_Ruddat_trop_Lag_CY}
Cheuk~Yu Mak and Helge Ruddat, \emph{Tropically constructed {L}agrangians in mirror quintic threefolds}, Forum Math. Sigma \textbf{8} (2020), Paper No. e58, 55. \MR{4179650}

\bibitem[Nho23]{Nho_spectral_network_family_fleor}
Yoon~Jae Nho, \emph{Family {F}loer theory, non-abelianization, and spectral networks}, preprint (2023), \href{https://arxiv.org/abs/2307.04213}{arXiv:2307.04213}.

\bibitem[OS24]{OS}
Yong-Geun Oh and Yat-Hin Suen, \emph{Lagrangian multi-sections and their toric equivariant mirror}, Adv. Math. \textbf{441} (2024), Paper No. 109545, 49. \MR{4710862}

\bibitem[Pay09]{branched_cover_fan}
S.~Payne, \emph{Toric vector bundles, branched covers of fans, and the resolution property}, J. Algebraic Geom. \textbf{18} (2009), no.~1, 1--36.

\bibitem[Sue21]{Suen_TP2}
Y.-H. Suen, \emph{Reconstruction of holomorphic tangent bundle of complex projective plane via via tropical {L}agrangian multi-section}, New York J. Math. \textbf{27} (2021), no.~2, 1096--1114.

\bibitem[Sue22]{Suen_TLMS_TLFS}
Yat-Hin Suen, \emph{Tropical {L}agrangian multi-sections and tropical locally free sheaves}, preprint (2022), \href{https://arxiv.org/abs/2203.02162}{arXiv:2203.02162}.

\bibitem[Sue23]{Suen_trop_lag}
\bysame, \emph{Tropical {L}agrangian multisections and toric vector bundles}, Pacific J. Math. \textbf{325} (2023), no.~2, 299--330. \MR{4662645}

\bibitem[SYZ96]{SYZ}
A.~Strominger, S.-T. Yau, and E.~Zaslow, \emph{Mirror symmetry is {$T$}-duality}, Nuclear Phys. B \textbf{479} (1996), no.~1-2, 243--259.

\bibitem[Tre17]{Morse_theory_TVB}
David Treumann, \emph{Morse theory and toric vector bundles}, Trans. Amer. Math. Soc. \textbf{369} (2017), no.~1, 1--29. \MR{3557766}

\end{thebibliography}

 \end{document}